\documentclass[12pt,a4paper]{amsart} 
\usepackage{amsmath}
\usepackage{amsfonts}
\usepackage{amssymb,tikz}
\usepackage{tikz-cd}
\usepackage[all,  curve, frame, arc, color]{xy}
\usepackage{amscd}
\usepackage{mathrsfs}
\usepackage{soul}
\usepackage{enumitem}

\usepackage{apptools}

\usepackage{hyperref}
\hypersetup{
    colorlinks,
    linkcolor={red!50!black},
    citecolor={blue!50!black},
    urlcolor={blue!80!black}
}

\usepackage{array}  

\usepackage{proofread}
\usepackage[textsize=tiny]{todonotes}

\usepackage{sidenotes}
\usepackage{adjustbox}

\usepackage{color,soul}

\usepackage{import}
\usepackage{xifthen}
\usepackage{pdfpages}
\usepackage{transparent}

\newcommand{\old}[1]{#1} 
\newcommand{\olc}[1]{{#1}} 
\def\FL{X_L^F}

\DeclareRobustCommand{\SkipTocEntry}[5]{}

\newcommand{\datogliere}[1]{} 

\newtheorem{Thm}{Theorem}[subsection]
\newtheorem{Lem}[Thm]{Lemma}
\newtheorem{Prop}[Thm]{Proposition}
\newtheorem{Cor}[Thm]{Corollary}

\theoremstyle{definition}
\newtheorem{df}[Thm]{Definition}
\newtheorem{choice}[Thm]{Choice}

\newtheorem{es}[Thm]{Example}
\newtheorem{Rem}[Thm]{Remark}

\AtAppendix{\counterwithin{Thm}{section}}

\newcommand{\Z}{\mathbb{Z}}
\newcommand{\CCC}{{\mathbb{C}}}
\newcommand{\R}{{\mathbb{R}}}
\newcommand{\Q}{{\mathbb{Q}}}

\newcommand{\B}{{\mathcal{B}}}

\newcommand{\Sc}{{\mathcal{S}}}
\newcommand{\Cc}{{\mathcal{C}}}
\newcommand{\Fc}{{\mathcal{F}}}
\newcommand{\Tc}{{\mathcal{T}}}

\newcommand{\rk}{{\operatorname{rk}}}

\newcommand{\GC}{\textstyle\int}
\newcommand{\Sal}{\operatorname{Sal}}

\newcommand{\into}[0]{\hookrightarrow}
\newcommand{\id}{\mathrm{id}}
\newcommand{\A}{{\mathcal{A}}}

\newcommand{\Ar}[1]{\A[#1]}
\newcommand{\gre}[1]{\vert #1 \vert}

\newcommand{\Hom}{\operatorname{Hom}}

\newcommand{\Aup}{\A^{\upharpoonright}}

\let\nsubset\relax

\usepackage{mathabx}

\newcommand{\op}{\operatorname{op}}

\makeatletter
\if@todonotes@disabled

\else

\fi
\makeatother

\begin{document}

\def\vg{\Lambda}	
\def\vd{\Omega}

\def\wl{\widehat{\lambda}}
\def\wg{\widehat{\omega}}

\def\T{\mathcal{T}}

\def\la{\langle}
\def\ra{\rangle}

\def\AL{\overline{\A}_{L}}
\def\AS{\overline{\A}_{S}}
\def\incl{\varphi}
\def\ideal{\mathcal{I}}

\def\G{\mathcal{G}}
\def\NN{\mathcal{N}}

\def\Stab{\operatorname{Stab}}
\def\supp{\operatorname{supp}}

\newcolumntype{C}{>{$}c<{$}}
\def\II{\mathcal{I}}
\def\JJ{\mathcal{J}}
\def\Gr{\operatorname{Gr}}
\def\tt{{\overline{t}}}

\newcommand{\ftodo}[1]{\marginnote{\begin{minipage}{3.1cm}\todo[inline]{#1}\end{minipage}}}

\title[Combinatorial representatives for toric arrangements]{Combinatorial generators for the cohomology of toric arrangements}
\author{Filippo Callegaro and Emanuele Delucchi}
\address[F. Callegaro]{
	Dipartimento di Matematica, University of Pisa, Italy.
}
\email{callegaro@dm.unipi.it}

\address[E. Delucchi]{University of Applied Arts and Sciences of Southern Switzerland.} 
\email{emanuele.delucchi@supsi.ch}

\begin{abstract}
We give a new combinatorial description of the cohomology ring structure of $H^*(M(\A);\Z)$ of the complement $M(\A)$ of a real complexified toric arrangement $\A$ in $(\CCC^*)^d$. In particular, we correct an 
error in the paper \cite{caldel17} .


\end{abstract}
\maketitle
\vspace{-3em}
\setcounter{tocdepth}{2}
\tableofcontents


\section{Introduction}\label{sec:algebra}


Research on the topology of configuration spaces gave rise to the study of a more general class of topological spaces, i.e., complements of arrangements of submanifolds. The general question in this field is to determine the topology of the complement of an arrangement of submanifolds in an ambient manifold, based on the combinatorial (i.e., incidence) data of the submanifolds themselves. The case of hyperplanes in real and complex vector spaces spurred a rich field of research. A main result on  complements of arrangements of hyperplanes is the presentation of the integer cohomology algebra by Orlik and Solomon \cite{os80} in terms of the arrangement's matroid. As combinatorial models for the homotopy types of hyperplane arrangement's complements have been developed, Orlik and Solomon's work has been enhanced by the study of explicit combinatorial generators of the integer cohomology of the arrangement's complement, see e.g., \cite{BjornerZiegler,FeichtnerZiegler,GorMacph}. In other words, starting from a cellular model for the arrangement's complement, one seeks to explicitly identify some subcomplexes representing cohomology classes that generate the cohomology groups (or the cohomology algebra) over the integers.

This paper focusses on arrangements of subtori in complex tori (see Section \ref{sec:setup} for a precise definition). Fresh impulse for this further chapter in the study of arrangements of submanifolds was given by work of De Concini and Procesi. Among other things, they obtained first results on the computation of the cohomology algebra of the complement. This 
program was followed up in the literature \cite{
dd2,bibby, dupont,a5,caldel17}. 
In particular, an explicit and general presentation of the rational and integer cohomology algebra of the complement to a toric arrangement was given in \cite{a5, a5corrigendum}, in terms of algebraic-combinatorial data of the arrangement. Moreover, for another bit of analogy with the theory of hyperplane arrangements, combinatorial cellular models for the complement of toric arrangements (so-called ``toric Salvetti complexes") have been developed in \cite{mocisette,DaDe1}. In \cite{caldel17} the attempt has been made to present the integer cohomology algebra of toric arrangements with explicit combinatorial generators given by subcomplexes of those models. However, some of the arguments given in \cite{caldel17} do not hold in the generality that is claimed in that paper.

Here we approach anew the problem of finding explicit combinatorial generators for the cohomology of (complexified) toric arrangements. We  give a constructive definition of two types of subcomplexes of the``toric Salvetti complex'' of a given complexified toric arrangement, and we show that they represent cohomology classes that generate the (integer) cohomology ring. Our results in this paper hold for general ``complexified" toric arrangements. 

In addition to correcting the flawed part of \cite{caldel17}, our construction should pave the way for an easier  explicit computational exploration of examples.

\bigskip

We start off outlining the main definitions and recalling some generalities about linear and toric arrangements - this is the content of Section \ref{sec:setup}. Then in Section \ref{sec:subcomplexes} we construct the subcomplexes of the toric Salvetti complex of a complexified toric arrangement that, in Section \ref{sec:cohom}, are shown to be representatives of generating classes for the arrangement complement's cohomology ring. We close by illustrating the construction and the computations with an example of an arrangement in the $2$-dimensional torus.

\section{Toric arrangements} \label{sec:setup}

We start by recalling some basics on the topology and combinatorics of toric arrangements. We will confine ourselves strictly to what is needed for the remainder of the paper. For a more thorough treatment (including context and more examples) we refer to \cite{DeCPlibro} (for general setup, \S\ref{setup1} and \S\ref{setup2}), \cite{BLSWZ} 
(for \S\ref{setup2}, \S\ref{setup4}), 
\cite{Salvetti} 
for \S\ref{setup5}, \cite[Appendix]{delpag}
 (for \S\ref{setupFC}) and \cite{caldel17} (for \S\ref{setup3} \S\ref{setup6}).

\subsection{Arrangements: linear and toric} \label{setup1}
An {\em arrangement} is a locally finite set $\A$ of submanifolds of a given manifold $X$. The {\em complement} of the arrangement is then the space
$$M(\A):=X \setminus \bigcup \A.
$$
We call $\A$ a ``real'' arrangement of hyperplanes if $X$ is $\mathbb R^d$ and $\A$ consists of affine hyperplanes, i.e., levelsets of linear forms. The {\em complexification} of such a real arrangement is the arrangement $\A_{\mathbb C}$ in $\mathbb C^d$ whose defining forms are the same as those for $A$.

In the case where the ambient space is $T\simeq(\mathbb C^*)^d$, a complex torus, we let a \emph{toric arrangement} $\A$ be a finite collection $\{ Y_1, \ldots, Y_n \}$, where for every $i$ we define $Y_i := \chi^{-1}(a_i)$, with $\chi_i \in \operatorname{Hom}(T,\mathbb{C}^*)$, $a_i \in \mathbb{C}^*$. We say that the arrangement $\A$ is \emph{complexified} if $a_i \in S^1$ for all $i$. In the latter case, the toric arrangement induces an arrangement in the compact torus $(S^1)^d$.
Note that the universal cover of the complex torus $T$ is a real Euclidean space, where we have  an arrangement of affine hyperplanes $\A^\upharpoonright$ obtained by lifting $\A$. The universal covering morphism of $T$ restricts to a map $q$ from the complement $M(\A^\upharpoonright)$ to $M(\A)$. 

\subsection{Layers and faces} \label{setup2}
Let $\A$ be either an arrangement of hyperplanes or a toric arrangement, and let $X$ denote the ambient space.
A {\em layer} of an arrangement $\A$ is any connected component of an intersection of elements of $\A$. We call $\Cc$ the set of all layers, partially ordered by reverse inclusion. The arrangement $\A$ is called {\em essential} if the maximal elements of $\Cc$ have dimension $0$. The poset $\Cc$ is ranked (by codimension of the layers, when $\A$ is essential) and we denote by $\Cc_i$ the set of elements of $\Cc$ with rank $i$.
To each $L\in \Cc$ we can associate the arrangement $\A_L=\{H\in \A \mid L\subseteq H\}$ in $X$ and the arrangement $\A^L=\{H\cap L \mid H\not\in \A_L\}$ determined by $\A$ in the space $L$.

\begin{Rem}
	In the case of an arrangement of hyperplanes, every nonempty intersection is connected. We thus often speak of ``intersections" or ``flats" when referring to layers. This terminology is motivated by the fact that, if the hyperplane arrangement is {\em central} (i.e., if all hyperplanes are linear subspaces), the poset of all intersections is the lattice of flats of a matroid.
\end{Rem}

\subsection{Real arrangements and sign vectors}\label{setup4} 
\def\H{\mathcal H}
Let $\H$ be a real arrangement of affine hyperplanes in $\mathbb R^d$.
 After choosing a ``positive side'' of each hyperplane we can associate to every point $x\in \mathbb R^d$   a {\em sign vector} $\gamma_x\in \{0,+,-\}^{\H}$ whose value on any hyperplane $H$ is $0,+,-$ according to whether $x$ lies on $H$, on the positive side of $H$ or on the negative side of $H$. A {\em face} of $\H$ is then the set of all $x\in \mathbb R^d$ with a fixed sign vector. The set $\Fc(\H)$ of all faces is partially ordered by inclusion of topological closures. The top-dimensional faces are called {\em chambers} and the set of chambers is denoted by $\Tc(\H)$.  The set of hyperplanes that separate two given chambers $C_1,C_2$ is denoted by $S(C_1,C_2)$. Any codimension-one face of a chamber $C$ is called a {\em wall} of $C$. A {\em gallery} is then an ordered sequence $C_1,W_1,C_2,\ldots$ where all $C_i$ are chambers and every $W_i$ is a wall of both $C_i$ and $C_{i+1}$.

 \begin{df} \label{df:flat_face}
 	Let $F \in \Fc(\H)$ a face of $\H$ and $X \in \Cc(\H)$ a layer. We define the face $F_X \in \Fc(\H_X)$ as the unique face such that for all $H \in \H_X$ $\gamma_{F_X}(H) = \gamma_{F}(H)$.
 	
 	For $F,G \in \Fc(\H)$, we define $F_G \in \Fc(\H)$ as the unique face such that
 	$$
 	\gamma_{F_G}(H)= \left\lbrace 
	\begin{array}{ll}
\gamma_F(H) & \mbox{if }\gamma_G(H) = 0;\\ \gamma_G(H) & \mbox{otherwise.}
	\end{array}
 	\right.
 	$$
 \end{df}

\subsection{Toric Arrangements and face category}\label{setupFC} Let $\A$ be a toric arrangement. 
We denote by $\Fc(\A)$ the {\em face category}  of $\A$, whose objects are all faces of the induced (polyhedral) cellularization of the compact torus and where  morphisms $F\to G$ correspond to the boundary cells of $G$ attached to $F$ (see \cite[Remark~3.3]{dd2}).
Given any $F\in \Fc(\A)$, there is a unique minimal layer containing $F$, called the support of $F$ and denoted by $\operatorname{supp}(F)$. 

\begin{Rem}\label{remark_polyhedral}
Note that the category $\Fc(\A)$ is the quotient (in the category of acyclic categories) of the poset $\Fc(\A^\upharpoonright)$ by the group action induced by the deck transformations of the covering $\mathbb C^d\to T$. In particular, $\Fc(\A)$ is the face category of a polyhedral CW complex in the sense of \cite{delpag} and its slice categories are posets of faces of polytopes (\cite[Lemma A.19]{delpag}).
\end{Rem}

\begin{Rem} Recall that to every acyclic category $\mathcal X$ is associated a cellular complex called the {\em nerve} $\gre{\mathcal X}$. If $\mathcal X$ is the face category of a polyhedral CW complex, then its nerve is homeomorphic to the complex itself.

We will follow established use and, throughout, talk about topological properties of an acyclic category (or a poset) referring to the topology of its nerve.
\end{Rem}

\subsection{Local hyperplane arrangements}\label{setup3} 
Let $\A$ denote a toric arrangement in the complex torus $T$.
To every face $F\in \Fc(\A)$
we associate a real arrangement of linear hyperplanes $\A[F]$, i.e., the real 
part of the (complexified) hyperplane arrangement defined by $\A$ in the tangent space to $T$ at any point in the relative interior of $F$. Moreover, associated to the toric arrangement $\A$ we consider an ``abstract'' arrangement of hyperplanes in $\mathbb R^d$ that we call $\A_0$, which can be thought of as the union (without repetitions) of all $\A[F]$ where $F$ ranges in $\Fc(\A)$. 
Key is the fact that, for every $F\in \Fc(\A)$,
$\A[F]$ is a subarrangement of $\A_0$. In particular, for every layer $L\in \Cc$ there is a subspace $X_L\in \Cc{} (\A_0)$ defined as the intersection of the hyperplanes associated to hypertori containing $L$. Given any $F\in\Fc(\A)$, we let $\FL$ be the smallest flat of $\A[F]$ containing $X_L$.


For every $F\in \Fc(\A)$ we will consider the poset of faces $\Fc(\A[F])$ and the set $\Tc(\A[F])$ of chambers of the (``local'') real arrangement $\A[F]$.
For every morphism $m:F\to G$ in $\Fc(\A)$ there is a natural inclusion $i_m:\Fc(\A[G])\to \Fc(\A[F])$ and in particular we call $F_m$ the image of the minimal element of  $\Fc(\A[G])$, see \cite[\S \olc{4.1}]{caldel17}. The map $i_m$ can be defined in terms of sign vectors.
\begin{df}\label{df:i_m} 
	 The map $i_m$ is determined as follows:
	$$
	\gamma_{i_m(K)}(H):= \left\lbrace 
	\begin{array}{ll}
	\gamma_{F_m} (H) & \mbox{if } H \notin \A[G]\\
	\gamma_K(H) & \mbox{else.}
	\end{array}
	\right.$$
\end{df}


\subsection{Affine Salvetti complex}\label{setup5}
Given a real 
arrangement $\H$ of affine hyperplanes, say in a real vector space $V,$ one can construct the associated Salvetti complex $\Sal(\H)$, which models the homotopy type of the complement of the complexification of $\H$  
(in the ambient - complex - vector space $V\otimes \mathbb C \simeq V +  iV$).
A natural construction of $\Sal(\H)$ is as the order complex of the partially ordered set
$$
\Sc(\H):=\{[G,C]\in \Fc(\H)\times \Tc(\H) \mid G\leq C\},
$$
 partially ordered via $[G,C]\geq [G',C']$ if $C_{G'}=C'$ (this means: no hyperplane containing $G'$ separates $C$ from $C'$, see 
 \cite{Salvetti}).


The gist now is that any choice of a point $p([G,C])$ in every open, convex subset $G +  iC$, where $[G,C]\in \Sal(\H)$, can be extended affinely to  a simplicial map $p: \Sal(\H) \to V\otimes \mathbb C $ (i.e., the simplex spanned by a chain $x_1<\ldots<x_k$ in $\Sal(\H)$ is mapped affinely to the simplex $\operatorname{conv}\{p(x_i)\}_i$ in $V\otimes \mathbb C$) and this extension defines an inclusion $\Sal(\H)\hookrightarrow M(\H_{\mathbb{C}})$ that is a homotopy equivalence.

Let us now go back  to the special arrangements of the type $\A[F]$, $F\in \Fc(\A)$, where $\A$ is a toric arrangement. For each chamber $C$ of $\A[F]$ we consider the subposet 
$$\Sc_C := \{ [G,K]\in \Sc(\A[F]) \mid  K=C_G\}$$
It will be useful to stratify $\Sc(\A[F])$ via the subposets
\begin{equation}\label{def_subposets_local}
\Sc^G(\A[F]) :=\bigcup_{C\geq G} \Sc_C,
\end{equation}
one for each $G\in \Fc(\A[F])$. For details on these constructions see \cite[\S 3.3]{caldel17}.

\subsection{Toric Salvetti complex}\label{dhcl}\label{setup6} 
Returning to the toric arrangement $\A$, a model for the complement of $M(\A)$ can be obtained from the following diagram. 
$$
\begin{array}{rccl}
\mathscr D: & \Fc(\A)^{op}& \to & \textrm{Posets} \\
 & F & \mapsto & \mathscr D(F):=\Sc(\A[F])\\
 & m: F\to G & \mapsto & \mathscr D(m):\Sc(\A[G])\to \Sc(\A[F])\\
 &&&[G,K]\mapsto [i_m(G),i_m(K)] 
\end{array}
$$
\begin{df}
Define the category $\Sc(\A):=\int\mathscr D$  (here $\int\mathscr D$ is the "Grothendieck construction" studied by Thomason \cite{Thomason}). Then, let $\Sal(\A)$ be the cell complex $\gre{\Sc(\A)}$, which can be realized as the quotient $q(\Sal(\Aup))$ and, thus, as a subset of $M(\A)$.
\end{df}

\begin{Prop}[{\cite{dd2}}]
 $\Sal(\A)$ is homotopy equivalent to $M(\A)$.  
\end{Prop}
Crucial to our discussion will be the following type of subcategories of $\int \mathscr D$. 
\begin{df}\label{df:subcatD}
For every $Y\in \Cc$ and every $F_0\in \Fc(\A_0)$ whose linear hull $\vert F_0\vert$ is $X_Y$, consider the subdiagram $\mathscr D_{Y,F_0}$ of $\mathscr D$ induced on the subcategory $\Fc(\A^Y)$ of $\Fc(\A)$ by the subposets $\mathscr D_{Y,F_0}(F):= \Sc^{F_0}(\A[F])$. In analogy to \cite[Definition \old{4.2.6}, Definition \old{4.2.8}]{caldel17} we set 
$$\Sc_{Y,F_0}:=\int \mathscr D_{Y,F_0}.$$
\end{df}
\begin{Rem} \label{htprod}
The homotopy type of $\Sc_{L,F_0}$ is that of  $|\Fc(\A^L)|\times \Sal(\A[L])$ 
(see 
\cite[Lemma \olc{4.2.15}]{caldel17}),
where $|\Fc(\A^L)|$ is the geometric realization of the order complex associated to the category $\Fc(\A^L)$.
\end{Rem}


	



\section{Combinatorial generators}
\newcommand{\newnot}[2]{\langle #1 : #2 \rangle}
\label{sec:subcomplexes}

In this section we aim at defining (cellular) representatives for certain homology classes of the complement $M(\A)$.

Intuitively, these classes will be of two types. The first type of classes, called $\wl_B^M$, will represent cycles that are parallel to $1$-dimensional layers $M$ and lie in the  subcomplexes 
$\gre{\Sc_{L,F_0}}$, with $M \subset L$, under certain conditions on $F_0$, $L$ and $B$ (see Remark \ref{rem:contenimenti} below). The second type are classes $\wg_H$, representing loops around a codimension-$1$ layer $H\in \A$.

\subsection{The generators  $\wl^M_B$}


\begin{df}\label{df:gallery}

Let $\A$ be an essential toric arrangement in a torus $T$ of dimension $d$.
For every 1-dimensional layer 
$M\in \Cc_{d-1}$ fix, once and for all, a chamber $R_M \in \Tc(\A_0)$ adjacent to 
$X_M$, and choose a minimal gallery in $\Tc(\A_0)$ 
$$R_M = C_0,C_1,\ldots, C_{k(M)} =\op_{X_M}(-R_M).$$ 
\end{df}


\begin{figure}[ht]
	\centering
		%
	\def\svgwidth{0.6\columnwidth}
\begingroup%
  \makeatletter%
  \providecommand\color[2][]{%
    \errmessage{(Inkscape) Color is used for the text in Inkscape, but the package 'color.sty' is not loaded}%
    \renewcommand\color[2][]{}%
  }%
  \providecommand\transparent[1]{%
    \errmessage{(Inkscape) Transparency is used (non-zero) for the text in Inkscape, but the package 'transparent.sty' is not loaded}%
    \renewcommand\transparent[1]{}%
  }%
  \providecommand\rotatebox[2]{#2}%
  \newcommand*\fsize{\dimexpr\f@size pt\relax}%
  \newcommand*\lineheight[1]{\fontsize{\fsize}{#1\fsize}\selectfont}%
  \ifx\svgwidth\undefined%
    \setlength{\unitlength}{457.88182536bp}%
    \ifx\svgscale\undefined%
      \relax%
    \else%
      \setlength{\unitlength}{\unitlength * \real{\svgscale}}%
    \fi%
  \else%
    \setlength{\unitlength}{\svgwidth}%
  \fi%
  \global\let\svgwidth\undefined%
  \global\let\svgscale\undefined%
  \makeatother%
  \begin{picture}(1,0.57236409)%
    \lineheight{1}%
    \setlength\tabcolsep{0pt}%
    \put(0,0){\includegraphics[width=\unitlength,page=1]{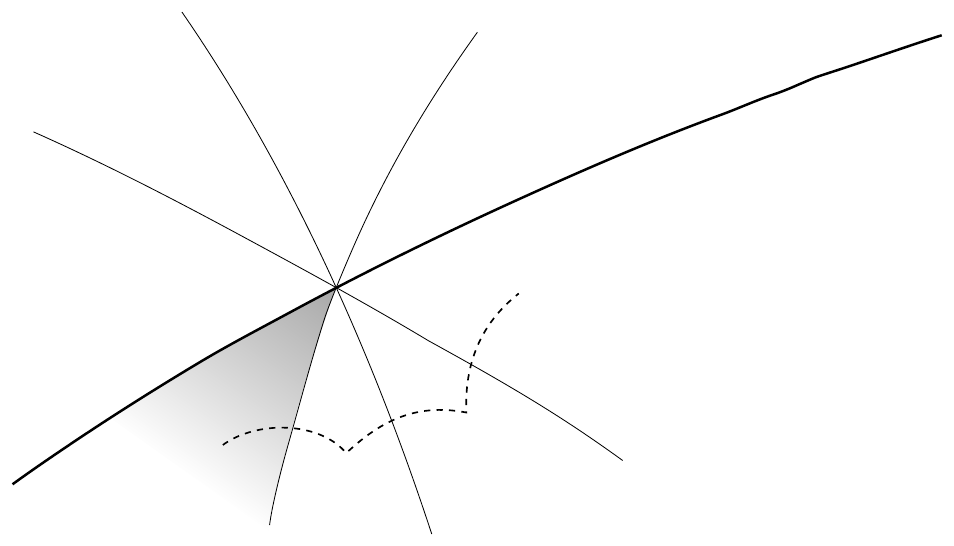}}%
    \put(0.0748013,0.05619978){\color[rgb]{0,0,0}\makebox(0,0)[lt]{\lineheight{1.25}\smash{\begin{tabular}[t]{l}$R_M=C_0$\end{tabular}}}}%
    \put(0.33500922,0.05306971){\color[rgb]{0,0,0}\makebox(0,0)[lt]{\lineheight{1.25}\smash{\begin{tabular}[t]{l}$C_1$\end{tabular}}}}%
    \put(0.47701203,0.11240224){\color[rgb]{0,0,0}\makebox(0,0)[lt]{\lineheight{1.25}\smash{\begin{tabular}[t]{l}$\ldots$\end{tabular}}}}%
    \put(0.55505245,0.26964131){\color[rgb]{0,0,0}\makebox(0,0)[lt]{\lineheight{1.25}\smash{\begin{tabular}[t]{l}$C_{k(M)}= \op_{X_M}(-R_M)$\end{tabular}}}}%
    \put(0.77766274,0.50215528){\color[rgb]{0,0,0}\makebox(0,0)[lt]{\lineheight{1.25}\smash{\begin{tabular}[t]{l}$X_M$\end{tabular}}}}%
    \put(0.89054898,0.15571031){\color[rgb]{0,0,0}\makebox(0,0)[lt]{\lineheight{1.25}\smash{\begin{tabular}[t]{l}$\mathcal{A}_0$\end{tabular}}}}%
    \put(0.5047714,0.44161607){\color[rgb]{0,0,0}\makebox(0,0)[lt]{\lineheight{1.25}\smash{\begin{tabular}[t]{l}$-R_M$\end{tabular}}}}%
  \end{picture}%
\endgroup%

  \caption{A minimal gallery $C_0, \ldots, C_{k(M)}$ in $\Tc(\A_0)$.}
	\label{fig:min_gal_1}
\end{figure}

Here, for every face $F\subseteq \mathbb R^d$ of $\mathcal A_0$ we write $-F$ for the negative of $F$ viewed as a set of vectors in $\mathbb R^d$. Moreover, if $C$ is a chamber adjacent to $X_M$, with $\op_{X_M}(C)$ we mean the chamber opposite to $C$ with respect to $X_M$, i.e., the unique chamber such that $\overline{C}\cap X_M=\overline{\op_{X_M}(C)}\cap X_M$ and $S(C,\op_{X_M}(C))=\mathcal A_{X_M}$. In particular, note that $S(C,-\op_{X_M}(C))=\mathcal A_0\setminus \A[M]$. For an illustration see Figure \ref{fig:min_gal_1}.


\datogliere{
\begin{Rem}\label{rem:alt1}
The choice of a different $R_M$, say $R_M'$, would give a different gallery, say $C_0',\ldots,C'_{k({M})}$. Let $\rho:=R_1,\ldots,R_h$ be a minimal gallery from $C_0$ to $C_0'$. Notice that, since $S(C_0,C'_0)\subseteq \A_{X_M}$, we have $S(C_0,C'_0)\cap S(C'_0,C'_{k(M)})=\emptyset$. Thus the concatenation of $\rho$ with $C_0',\ldots,C'_{k({M})}$ is a minimal gallery, as is the concatenation of $C_0,\ldots,C_{k({M})}$ with $\tilde{\rho}:=\op_{X_M}(-\rho)$.

\end{Rem}
}

\begin{df}

Given any chamber $C\in \Tc(\A_0)$ and any $F\in \Fc(\A)$ we denote by $\newnot{C}{F}$ the unique chamber of $\A[F]$ containing $C$.

\end{df}	

Given any face $F\subseteq M$, let 
$$
C_0^F,\ldots,C_{k(M,F)}^F
$$
be an enumeration of the set $\{\newnot{C_i}{F}\}_{i=0,\ldots,k(M)}$
in increasing index order and call $W_i^F$
the wall separating $\newnot{C_i}{F}$
from $\newnot{C_{i+1}}{F}$ in $\A[F]$. This is illustrated in Figure \ref{fig:min_gal_2}.

\begin{figure}[htb]
	\centering
    %
	\def\svgwidth{0.5\columnwidth}
	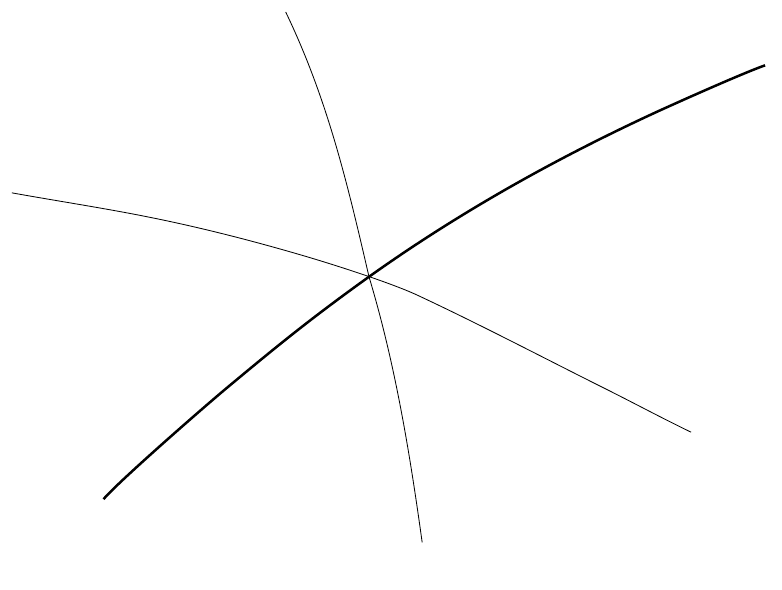

	\caption{The gallery $C_0^F,\ldots,C_{k(M,F)}^F$
	in $\A[F]$.}
	\label{fig:min_gal_2}
\end{figure}


\begin{Rem}\label{rem:mini}
The sequence $C^F_0,W_1^F,C_1^F,\ldots$ defines a minimal gallery in $\A[F]$, hence it never crosses any hyperplane in $\A[M]$. 
In terms of sign vectors, $\gamma_{C_i^F}(H)=\gamma_{W_i^F}(H)=\gamma_{R_M}(H)$
for all $i$ and all $H\in \A[F]\cap\A[M]$ (see \cite[\S 3.2.1]{caldel17}).
\end{Rem}

\def\pth{\operatorname{Path}}
\begin{df}\label{defpath}
For every $B\in \Tc(\A_0)$, set 
\begin{center}
$v_i(M;B,F):=
[C_i^F,C_i^F]$ for $i=0,\ldots,k(M,F)$, \\ $e_i(M;B,F):=[W_{i+1}^F,\newnot{B}{F}_{W_{i+1}^F}]$ for $i=0,\ldots,k(M,F)-1$.
\end{center}
Define the following subposet of $\Sc(\A[F])$: 
$$
\pth(M;B,F):=\{
v_i(M;B,F)\}_{i=0,\ldots,k(M,F)}
\cup
\{e_i(M;B,F)\}_{i=0,\ldots,k(M,F)-1}.
$$
\end{df}

\begin{figure}[htb]
	\centering
    %
	\def\svgwidth{0.7\columnwidth}
	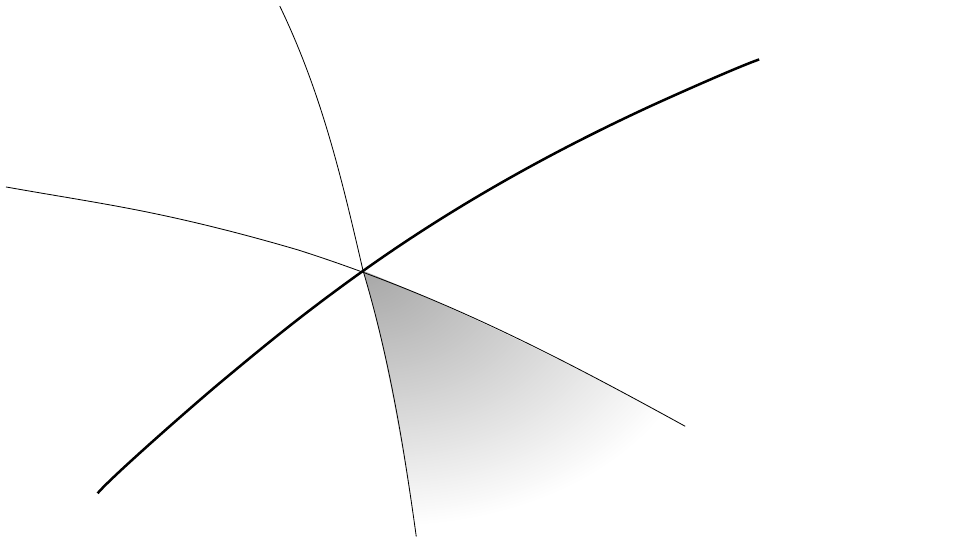

	\caption{The poset $\pth(M;B,F)$.}
	\label{fig:min_gal}
\end{figure}


Note that the poset $\pth(M;B,F)$  has the following form:
\begin{center}{\small
\begin{tikzpicture}
\node (A) at (-4,0)  {$v_0(M;B,F)$};
\node (B) at (-2,1)  {$e_0(M;B,F)$};
\node (C) at (0,0)  {$v_1(M;B,F)$};
\node (D) at (2,1)  {$e_1(M;B,F)$};
\node (E) at (4,0)  {$v_2(M;B,F)$};
\node at (6,0)  {$\ldots$};
\node at (6,1)  {$\ldots$};
\draw[->] (A) -- (B);
\draw[->] (C) -- (D);
\draw[<-] (B) -- (C);
\draw[<-] (D) -- (E);
\end{tikzpicture}
}
\end{center}
so that its topological realization $\gre{\pth(M; B,F)}$ is a topological path from $v_0(M;B,F)$ to $v_{k(M,F)}(M;B,F)$ in the Salvetti complex of $\A[F]$.
\begin{Rem}\label{rem:dim1}
If $\dim(F)=1$, then $k(M,F)=0$ so $\pth(M;B,F)$ in this case is a single vertex which we will denote $v(M;B,F)$. 
\end{Rem}

\begin{Rem}\label{rem:eebar}
For any two $B\neq B'$ we have
$v_i(M;B',F) = v_i(M;B,F)$ for all $i$ and
$e_i(M;B',F) = e_i(M;B,F)$ if and only if the affine span of $W_i$ does not separate $B$ from $B'$ (hence $\newnot{B}{F}$ from $\newnot{B'}{F}$). If we set
$$\overline{e}_i(M;B,F):=[W_i^F,(-\newnot{B}{F})_{W_i^F}]$$
we can state more precisely
\begin{equation}\label{extremis_eq}
e_i(M;B',F) = \left\{\begin{array}{ll}
e_i(M;B,F) & \textrm{if } W_i^F \not\in S(B,B'),\\
\overline{e}_i(M;B,F) & \textrm{otherwise.}
\end{array}\right.
\end{equation}
\end{Rem}

For every $H\in \A_F \setminus \A_M$ 
there is a unique $i$ such that $W_i^F= X_H$,
thus we can define a subcategory $\Xi(H;B,F)$:
\begin{equation}\label{def:Xi}
\begin{tikzpicture}[y=2em,x=3em,baseline=(current  bounding  box.center)]
\node (A) at (-1,1) {$e_i(M;B,F)$};
\node (B) at (1,1) {$\overline{e}_i(M;B,F)$};
\node (C) at (1,-1) {$v_{i+1}(M;B,F)$};
\node (D) at (-1,-1) {$v_i(M;B,F)$};
\draw [->] (C) -- (A);
\draw [->] (C) -- (B);
\draw [->] (D) -- (A);
\draw [->] (D) -- (B);
\end{tikzpicture}
\end{equation}

\begin{df}\label{def:Lambda}
Let $\A$ be an essential toric arrangement in a torus $T$ of dimension $d$. 
For every $1$-dimensional layer $M\in \Cc_{d-1}$  and every chamber $B\in \T(\A_0)$, define the induced subcategory of $\Sc(\A)=\int\mathscr D$  on the object set
$$\vg_B^M: = \bigcup_{F \subseteq M}
\{(F, X) \mid X\in \pth(M;B,F)\}.
$$
\end{df}
In order to understand the structure of the subcategory $\vg_B^M$ let us first consider the category $\Fc(\A^M)$. 
Since it is the face category of a polyhedral cellularization of $S^1$ (cf.\ Remark \ref{remark_polyhedral}),
every object $P$ of $\Fc(\A^M)$ of dimension $0$ is the origin of two arrows and every object $G$ of dimension $1$ is the target of two arrows. Choose an object $P_0$ of dimension $0$ and consider the two arrows, say $m_1,m_2$ , originating in $P_0$. Then $F_{m_1} = - F_{m_2}$ in $\Fc(\A[P_0])$ and in particular exactly one of these -- say, $F_{m_2}$ -- is adjacent to $-\newnot{B}{P_0}$. Call $G_0$ the target of $m_2$, and call $P_1$ the origin of the other nontrivial morphism ending in $G_0$. Continuing this way we can naturally label the objects of $\Fc(\A^M)$ as
\begin{equation}\label{catenina}
P_0 \rightarrow G_0 \leftarrow P_1 \rightarrow G_1 \leftarrow \ldots \rightarrow G_{\ell(M)} \leftarrow P_{0}.
\end{equation}
 \begin{Lem}
 In the category $\Sc(\A)=\int \mathscr D$ we have, for all $i$ modulo $\ell(M)$, arrows
 $$ (G_j, v(M;B,G_j) )
 \rightarrow 
 (P_i, v_l(M;B,P_{i})) $$
 if and only if either $j=i-1$ and $l=0$, or else $j=i$ and $l=k(M,P_i)$. (The index-less $v(M;B,G_j)$ is $[C_0^{G_{j}},C_0^{G_{j}}]$, the only element of $\pth(M;B,G_j)$, as in Remark \ref{rem:dim1}.)
 \end{Lem}
 \begin{proof}
 Let us consider a piece $G_{i-1} \stackrel{m_1}{\leftarrow} P_i \stackrel{m_2}{\to} G_{i}$ of the chain of Diagram \eqref{catenina}. We refer to Figure \ref{fig:category} for illustration. 
 \begin{figure}[ht]
 	\centering
	\def\svgwidth{1\columnwidth}
	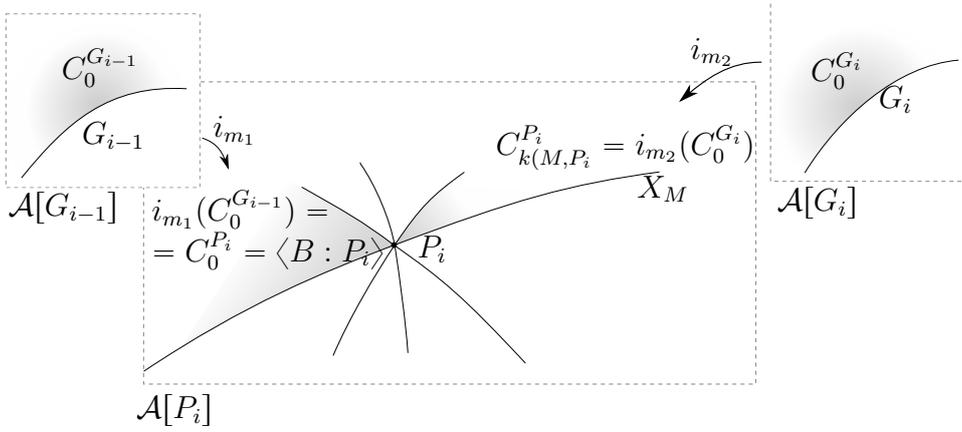

 	\caption{A picture of $\int \mathscr D$ near $G_{i-1}, P_i, G_i$ with arrows $i_{m_1}, i_{m_2}$.}
 	\label{fig:category}
 \end{figure}
 By construction, $i_{m_1}(G_{i-1})$ is adjacent to $\newnot{B}{P_{i}}$ and $i_{m_2}(G_{i})$ is adjacent to $-\newnot{B}{P_{i}}$ . Thus, the associated functions
 $$
 \Fc(G_{i-1}) \stackrel{i_{m_1}}{\rightarrow} \Fc(P_i) \stackrel{i_{m_2}}{\leftarrow} \Fc(G_{i})
 $$
 have  $i_{m_1}(C_0^{G_{i-1}}) = C_0^{P_{i}}$ and $i_{m_2}(C_0^{G_{i}}) = C_{k(M,P_{i})}^{P_{i+1}}$. Thus $\mathscr D(m_2)(v(M;B,G_{i}))=v_{k(M,P_{i})}(M;B,P_{i})$ and $\mathscr D(m_1)(v(M;B,G_{i-1}))=v_{0}(M;B,P_{i})$. This means that $(m_1,\mathscr D(m_1))$ and $(m_2,\mathscr D(m_2))$ are the claimed morphisms in $\int\mathscr D$. The fact that $\pth(M;B,G_{i-1})$ and $\pth(M;B,G_i)$ have only one element implies that there are no morphisms to other vertices $v_l(M;B,P_i)$'s.
 \end{proof}
 \begin{Cor}
The category $\vg_B^M$ is of the form
\begin{center}{\tiny
\begin{tikzpicture}[x=3.2em]
\node (A) at (-3.5,0)  {$(P_0,v_0(M;B,P_0))$};
\node (B) at (-1.5,1)  {$(P_0,e_0(M;B,P_0))$};
\node (B1) at (-1.5,0) {};
\node (D) at (2,1)  {$(P_0,e_{k(M,P_0)}(M;B,P_0))$};
\node (E) at (4,0)  {$(P_0,v_{k(M,P_0)}(M;B,P_0))$};
\node (F) at (6,-1)  {$(G_0,v(M;B,G_0))$};
\node (G) at (8,0)  {$(P_1,v_{0}(M;B,P_1))$};
\node (H) at (10,1)  {$(P_1,e_{0}(M;B,P_1))$};
\node at (11,0.5)  {$\ldots$};
\node at (0,.5)  {$\ldots$};
\draw [->] (A) -- (B);
\draw [<-] (D) -- (E);
\draw [<-] (E) -- (F);
\draw [->] (F) -- (G);
\draw [->] (G) -- (H);
\end{tikzpicture}
}
\end{center}
In particular, it is a poset homeomorphic to $S^1$.
\end{Cor}

\begin{Rem}\label{rem:FundChoices} 
A different choice of the gallery in Definition \ref{df:gallery} does not affect the homotopy type of the subcomplex $\gre{\vg_{R_M}^M}$ inside the toric Salvetti complex $\Sal(\A)$.
The same is true if one changes $R_M$ to any other chamber $R'_M$ with $R_M\cap X_M = R'_M\cap X_M$. Moreover, choosing the same gallery but traversed in the opposite sense will not change the subcomplexes $\Lambda_{\ast}^M$.
\end{Rem}

\begin{Rem} \label{rem:contenimenti}
If $M$ is a $1$-dimensional layer contained in $L$, and $B$ is any chamber of $\A_0$, we have $\vg_B^M\subseteq \mathcal S_{L,F_0}$ if $F_0=\overline{B}\cap X_L$. 
\end{Rem}

 \begin{df} \label{df:lambdini}
 For $B\in \Tc(\A_0)$ and $M$ any layer of dimension 1, let 
 $
 [\vg_B^M] \in C_1(\Sal(\A))
 $ denote the cycle supported on $\gre{\vg_B^M}$ and uniquely determined by setting the coefficient of $(v_0(M;B,P_0) \to e_0(M;B,P_0))$ equal to $1$.

Then, let $
\wl_B^M 
$ 
be the homology class of
$ [\Lambda_B^M]
$.

 \end{df}

For the following ``basis-change'' formula we need to define, for any $F\in \Fc(\A)$ and any two chambers $B,B'\in \Tc(\A_0)$, the set 
\begin{equation}
S_F(B,B'):=\{H\in \A_F\mid X_H \textrm{ separates }B\textrm{ from }B'\}. 
\end{equation}

\begin{Prop}\label{prop:cambio_camera}
Let $B,B'\in \Tc(\A_0)$. Then
$$
[\vg_B^M] - [\vg_{B'}^M] = \sum_{P \subseteq M} \sum_{
H \in  S_P(B,B')\setminus \A_M
} 
[\Xi (H;B,P)]
$$
where $[\Xi(H;B,P)]$ is the $1-$cycle determined by the geometric realization of the subcategory defined in Diagram \eqref{def:Xi} with the orientation given by setting the coefficient of $v_i(M;B,P)\to e_i(M;B,P)$ equal to $1$.
\end{Prop}


 \begin{proof}
 We write the elementary chain corresponding to a morphism $m:F\to G$ as $[F\to G]$ and let its boundary be $[G]-[F]$. Moreover, write $v_i(P)$, $e_i(P)$ for $v_i(M;B,P)$, $e_i(M;B,P)$ and  $v_i'(P)$, $e_i'(P)$ for $v_i(M;B',P)$, $e_i(M;B',P)$. In this way, recalling Equation \eqref{def:Xi} we can write
  \begin{align*}
 [\Xi (H;B,P)] = [v_i(P) \to e_i(P)] &- [v_{i+1}(P) \to e_{i}(P)]\\ &+[v_{i+1}(P) \to \overline{e}_{i}(P)]-[v_i(P) \to \overline{e}_i(P)]
 \end{align*}
 where $W_i^P= X_H$, and the difference of chains $[\vg_B^M] - [\vg_{B'}^M]$ is
\begin{align*}
\sum_{P \subseteq M} \sum_{i=0}^{k(M,P)-1} 
[v_i(P) \to e_i(P)] &- [v_{i+1}(P) \to e_{i}(P)]\\
&-([v_i'(P) \to e_i'(P)] - [v_{i+1}'(P) \to e_{i}'(P)] )\\
= \sum_{P \subseteq M} \sum_{i: W_i^P \in S_P(B,B')} 
[v_i(P) \to e_i(P)] &- [v_{i+1}(P) \to e_{i}(P)] \\
&+[v_{i+1}(P) \to \overline{e}_{i}(P)]-[v_i(P) \to \overline{e}_i(P)] \\
=\sum_{P \subseteq M} \sum_{H \in S_P(B,B')} 
& [\Xi (H;B,P)]
\end{align*}
 where, for the second equality, we used Remark \ref{rem:eebar} and the fact that $e_i'(P)={e}_i(P)$ whenever $W_i^P$ does not separate $B$ from $B'$, otherwise $e_i'(P)=\overline{e}_i(P)$ (see Equation \eqref{extremis_eq}).
 
 \end{proof}

%

\subsection{The generators $\wg_H$}

\begin{df}
For every $H\in \A$ choose, once and for all, a chamber ${R_H\in \Tc(\A_0)}$. 
For every $m:F\to G\in \operatorname{Mor}\Fc(\A^H)$ with $\operatorname{supp}(G)=H$,
let $C_1:= \newnot{R_H}{F}_{F_m}$, 
$C_2:=\newnot{-R_H}{F}_{F_m}$ 
be the two chambers of $\Tc(\A[F])$ adjacent to $F_m$. We define a subcategory $\Omega^{(m)}$ of $\Sc(\A)$ as 

\begin{center}{\small
\begin{tikzpicture}[x=4em]
\node (L) at (-4,.5) {$\Omega^{(m)}:=$};
\node (A) at (-2,0)  {$(F,[C_1,C_1])$};
\node (B) at (-2,1)  {$(F,[F_m,C_1])$};
\node (C) at (0,0)  {$(F,[C_2,C_2])$};
\node (D) at (0,1)  {$(F,[F_m,C_2])$};
\node (R) at (2,.5) {$\subseteq\Sc(\A).$};
\draw (A) -- (B) -- (C) -- (D) -- (A);
\end{tikzpicture}
}
\end{center}

\end{df}

\begin{Rem}\label{rem:OmegaId}
When $m = id_G\, : G\to G$, $\Omega^{(id_G)}$ is the  poset 
\begin{center}{\small
\begin{tikzpicture}[x=4em]
\node (L) at (-4,.5) {$\Omega^{(id_G)}=$};
\node (A) at (-2,0)  {$(G,[C_1,C_1])$};
\node (B) at (-2,1)  {$(G,[G,C_1])$};
\node (C) at (0,0)  {$(G,[C_2,C_2])$};
\node (D) at (0,1)  {$(G,[G,C_2])$};
\node (R) at (2,.5) {$\subseteq\Sc(\A).$};
\draw (A) -- (B) -- (C) -- (D) -- (A);
\end{tikzpicture}
}
\end{center}

\end{Rem}

\begin{df}\label{df:Om}
Let $[\Omega^{(m)}]$ denote the $1-$cycle of $\Sal(\A)$ supported on $\gre{\Omega^{(m)}}$ uniquely determined by the orientation given by setting the coefficient of $((G,[C_1,C_1]) \to (G,[G,C_1]))$ equal to $1$.
\end{df}


%
%

\begin{df}
Let $H\in \A$ and $B \in \Tc(\A_0)$. 
Define

$$ \epsilon(H,B):=
\left\{\begin{array}{ll}
1 & \textrm{ if }  X_H\not\in S(R_H, B)\\
-1 & \textrm{ if } X_H\in S(R_H,B)
\end{array}\right.$$
\end{df}

\begin{Lem}\label{lem:relations_homology} We have the following relations in the homology of $\Sal(\A)$.
$\,$
\begin{itemize}
\item[(i)] For every $m:F\to G$, $[\Omega^{(m)}]\simeq [\Omega^{(id_G)}]$.

\item[(ii)] $[\Omega^{(F\to G)}]\simeq [\Omega^{(F\to G')}]$ when both $G,G'$ are of maximal dimension in $\Fc(\A^H)$.

\item[(iii)] $[\Omega^{(P\to W_i^P)}] = \epsilon (H,B)[\Xi(H;B,P)]$ if $\vert W_i^P \vert = X_H$. 
\end{itemize}

\end{Lem}

\begin{proof}
First notice that if $m:F\to G$ with $G$ of codimension $1$, then 
$$\Omega^{(m)}=\mathscr D (m)(\Omega^{id_G}).$$
In particular the complex $\Sal(\A)$, being obtained via a homotopy colimit as in \S\ref{dhcl}, contains the mapping cylinder  of $\mathscr D (m)\vert_{\Omega^{(m)}}$ in the form of the nerve of the subcategory on the right-hand side of Figure \ref{figxtra}.
\begin{figure}[h]
{\small
\begin{tikzpicture}[x=4em]
\node (Lm) at (-4,.5) {$\Omega^{(m)}$};
\node (Am) at (-2,0)  {$(F,[C_1,C_1])$};
\node (Bm) at (-2,1)  {$(F,[F_m,C_1])$};
\node (Cm) at (0,0)  {$(F,[C_2,C_2])$};
\node (Dm) at (0,1)  {$(F,[F_m,C_2])$};
\draw (Am) -- (Bm) -- (Cm) -- (Dm) -- (Am);
\node (Li) at (-4,2.5) {$\Omega^{(id_G)}$};
\node (Ai) at (0,2)  {$(G,[C_1,C_1])$};
\node (Bi) at (0,3)  {$(G,[G,C_1])$};
\node (Ci) at (2,2)  {$(G,[C_2,C_2])$};
\node (Di) at (2,3)  {$(G,[G,C_2])$};
\draw (Ai) -- (Bi) -- (Ci) -- (Di) -- (Ai);
\draw (Am) -- (Ai);
\draw (Bm) -- (Bi);
\draw (Cm) -- (Ci);
\draw (Dm) -- (Di);
\draw[->] (Lm) -- (Li);
\node[anchor=east] (jm) at (-4,1.5) {$\mathscr D (m)$};
\end{tikzpicture}
}
\caption{The morphism $\mathscr D (m)\vert_{\Omega^{(m)}}$ and the subposet defining its mapping cylinder.} \label{figxtra}
\end{figure}
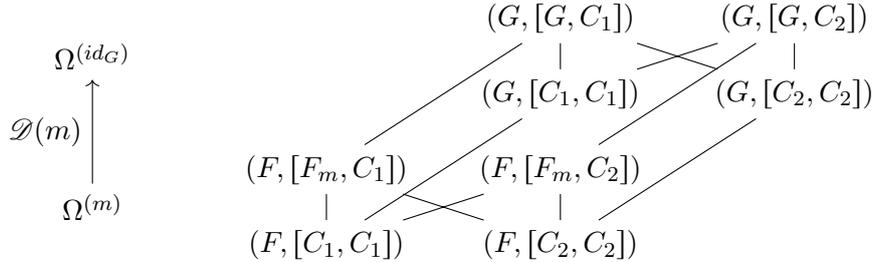
This mapping cylinder gives a homotopy inside $\Sal(\A)$ between $\gre{\Omega^{id_G}}$ and $\gre{\Omega^{(m)}}$ that sends edges to ``corresponding edges''. Thus (i) follows.

Part (ii) follows analogously by a homotopy between the two subcomplexes inside the subcomplex $\gre{Sc^G(\A[F])}=\gre{\Sc^{G'}(\A[F])}$  of $\Sal(\A)$ (see the discussion around Proposition \old{3.3.5} in \cite{caldel17}).

For part (iii) notice that, for any $B$, $\Omega^{(P\to W_i^P)}$ and $\Xi(H;B,P)$ are the same subposet. The associated chains differs by a sign depending on whether $B$ is on the same side of $X_H$ 
as $R_H$.
\end{proof}

\begin{Cor} \label{cor:wg}
For every $H\in \A$ the homology class of any $[\Omega^{(m)}]$ does not depend on the choice of $m:F\to G$ as long as $\operatorname{supp}(G)=H$. 
\end{Cor}

The following definition is now well-posed.

\begin{df}\label{def:omegahat}
For every $H\in \A$ let us denote  by $$\wg_H\in H_1(\Sal(\A),\mathbb Z)$$ the homology class of (any) $[\Omega^{(m)}]$ with $m:F\to G$ and $\operatorname{supp}(G)=H$.
\end{df}



We close with a formula relating the classes $\wg_H$ with the $\wl_B^M$ defined in  Definition \ref{df:lambdini}.
\begin{Prop}\label{prop:cambioL}
%

\begin{equation}\label{eq:lambda_cap}
\wl^M_B - \wl^M_{B'} = 
\sum_{\substack{
F\subseteq M \\ H\in S_F(B,B')\setminus \A_M 
}} \epsilon(H,B) \wg_H.
\end{equation}
\end{Prop}

\begin{proof}
 Lemma \ref{lem:relations_homology} allows us to eliminate the dependency on $P$ in the right-hand side of the equality of Proposition \ref{prop:cambio_camera}, and to rewrite it as in this Proposition's claim.
\end{proof}
\begin{Rem} With Remark \ref{rem:FundChoices}, the generators $\wl_{R_M}^M$ are independent on the choices in Definition \ref{df:gallery}. Proposition \ref{prop:cambioL} then shows how to recover all $\wl_B^M$ from the $\wl_{R_M}^M$. Thus our system of generators does not depend on the choices in Definition \ref{df:gallery}
\end{Rem}
\section{Cohomology of quotients and subcomplexes}
\label{sec:cohom}
\subsection{Quotients of toric arrangements}

\begin{df}
Let $\A$ be a toric arrangement and let $L \in \Cc(\A)$ be a layer. 
We consider $T$ as a group and we define $L_0$ as the subgroup of $T$ that has $L$ as a coset.
Recall that $\A_L$ is the subarrangement of $\A$ given by the hypertori that contain $L$. Consider the arrangement $$\AL:=\A_L/L_0:=\{H/L_0\mid H \in \A_L\} \textrm{ in } T/L_0.$$
We define the quotient map 
$$ 
f_L: M(\A) \to M(\AL)
$$ as the composition $\pi_{L_0} \circ i_L$
 of the inclusion 
$$
i_L: M(\A) \to M(\A_L)
$$
and the projection
$$
\pi_{L_0}: M(\A_L) \to M(\AL).
$$
\end{df}


%

\begin{df} The quotient by $L_0$ induces order-preserving maps
$$
\pi_L:\Cc(\A) \to \Cc(\AL)
\textrm{ and } 
\pi_L: \Fc(\A) \to \Fc(\AL)$$
and the latter map lifts to the natural order-preserving map
$$\Fc(\A^{\upharpoonright})\to \Fc ((\A_L)^{\upharpoonright}), \quad\quad F\mapsto \min_{\subseteq} \{G\in \Fc((\A_L)^{\upharpoonright})\mid F\subseteq G\}$$
sending every face to the smallest one that contains it.

\end{df}

\begin{Rem} \label{rem:quisquilie} We note two elementary facts about sign vectors that can be gathered directly from  \cite[Definition~\old{3.2.1}]{caldel17}. Let notation be as in Definition \ref{df:flat_face} and \ref{df:i_m} and Section \ref{setup3}. In particular, recall that  for any $F\in\Fc(\A)$, $\FL$  is the smallest flat of $\A[F]$ containing $X_L$.
\begin{itemize}
\item[(1)]
For all $m\in \operatorname{Mor}{\Fc(\A)}$ with source object $K$ and every $H\in \A[K]_{\FL}$ we have 
that
$
\gamma_{F_{\pi_L(m)}}(H/L_0) = \gamma_{F_m}(H).
$ 
\item[(2)] For all $G,K\in \A[F]$ and every flat $X$ of $\A[F]$ we have  $(G_X)_{(K_X)}=(G_K)_X$.
\end{itemize}
\end{Rem}
We see that the linear arrangement $(\AL)_0$ is $(\A_0)_{X_L}/X_L$, and in particular we have a natural map
$$
\pi_L: \Fc(\A_0)\to \Fc(\A_0)_{X_L}\simeq\Fc((\AL)_0), \quad F\mapsto F_{X_L}.
$$
Similarly, for every $F\in \Fc(\A)$, the arrangement $\AL[\pi_L(F)]$ is the essentialisation (see \cite[Lemma~5.30]{OT}) of the sub-arrangement of $\A[F]_{\FL}\subseteq \A[F]$ consisting of all hyperplanes containing $\FL$. Thus the map
$$
\pi_L^F: \Fc(\A[F]) \to \Fc(\AL[\pi_L(F)]), \quad K\mapsto K/X_L
$$
is order preserving and surjective, and restricts to an isomorphism of posets
$$\Fc(\A[F]_{\FL}) \longrightarrow \Fc(\AL[\pi_L(F)]). $$ 
Thus we can identify $\Sc(\AL[\pi_L(F)])$ with $\Sc(\A[F]_{\FL})$ and we have the following natural order-preserving map (see \cite[Definition~\olc{3.3.2}]{caldel17})
$$
b_{\FL}: \Sc(\A[F]) \to \Sc (\A_L[\pi_L(F)]), \quad [K,C] \mapsto 
[K_{\FL},C_{\FL}]
$$

\begin{Lem}\label{NatTra} For all layers $L$,
$$
(\pi_L, b_{X_L^{\ast}}) : \mathscr D(\A) \Rightarrow \mathscr D(\AL)
$$
is a natural transformation.
\end{Lem}

\begin{proof}[Proof of Lemma \ref{NatTra}] 
In order to check naturality pick any $m: F\to G$ in $\operatorname{Mor}(\Fc(\A))$ and $[K,C]\in \Sc(\A[G])$. According to the definitions given above we have the following diagram:

\begin{center}
	\begin{tikzcd}
(F,[i_m(K),i_m(C)]) 
\arrow[r, "(\pi_{L}{,}b_{\FL})"] & 
(\pi_L(F),[(i_{m}(K))_{\FL},(i_{m}(C))_{\FL}] ) 
\arrow[d,<->,"\stackrel{?}{=}"]\\
& (\pi_L(F),
[i_{\pi_L(m)}(K_{X^G_L}),
i_{\pi_L(m)}(C_{X^G_L})] )\\
(G,[K,C]) 
\arrow[r, "(\pi_{L}{,}b_{X^G_L})"]
\arrow[uu,"(m{,}j_m)"] &  
(\pi_L(G),[K_{X^G_L},C_{X^G_L}] )
\arrow[u,"(\pi_L(m){,}j_{\pi_L(m)})"]
\end{tikzcd}
\end{center}
and we need to prove equality of the two expressions on the top right-hand-side. It is enough to prove that, for every $K\in \Fc(\A[G])$,
\begin{equation}\label{eq:sign_vec}
i_{\pi_L(m)}(K_{X^G_L}) = (i_{m}(K))_{\FL} 
\quad \textrm{ in }
\A[F]_{\FL}.
\end{equation}
This we do using the definition of $\gamma_{i_m(K)}$ as in \cite[Remark~\olc{4.1.1}]{caldel17}. First consider the right-hand side of Equation \eqref{eq:sign_vec}, that it is defined by 
\begin{equation}\label{eq:idL1}
\gamma_{(i_{m}(K))_{\FL}}(H)=
\gamma_{i_{m}(K)}(H)= 
\left\{\begin{array}{ll}
\gamma_{i_m(K)}(H) & H\not\in\A[G]\\
\gamma_K(H) & H\in \A[G]
\end{array}\right.
\end{equation}
for all $H\in \A[F]_{\FL}$.
In the same vein, the left-hand side of Equation \eqref{eq:sign_vec} is determined by 
\begin{equation}\label{eq:idL2}
\gamma_{i_{\pi_L(m)}(K_{X^G_L})}(H)=
\left\{\begin{array}{ll}
\gamma_{F_{\pi_L(m)}}(H)=\gamma_{F_{m}}(H) &\textrm{if }H\not\in \A[G]\\
\gamma_{K_{X^G_L}}(H)=\gamma_{K}(H) &\textrm{if }H\in \A[G]
\end{array}\right.
\end{equation}
where we used Remark \ref{rem:quisquilie}.(1). Now, with  \cite[Remark~\olc{4.1.1}]{caldel17} we see that 
$\gamma_{F_m}(H)=\gamma_{i_m(K)}(H)$ for  $H\not\in\A[G]$,
completing the check of the identity between the two sides of the required equality, as expressed in Equations \eqref{eq:idL1} and \eqref{eq:idL2}.
\end{proof}

\begin{df}\label{df:map_Phi} The natural transformation of Lemma \ref{NatTra} induces a functor
$$
\Phi_L: \Sc(\A) \to \Sc(\AL)
$$
and thus a (cellular) map
$$
\Phi_L: \Sal(\A) \to \Sal(\AL).
$$
We slightly abuse notation by using the same symbol for the two maps. The distinction will be apparent from the context.
\end{df}
\begin{Lem}\label{lem:Phi} Let $\A$ be an essential toric arrangement in a torus $T$ of dimension $d$. Fix a layer $L \in \Cc$. 
\begin{itemize}[leftmargin=1.5em]
\item[(1)] For every $M\in \Cc_{d-1}$ such that $X_M\subseteq X_L$ 
and every chamber $B\in \Tc(\A_0)$, $\Phi_L(\Lambda^M_B)$ is a single vertex. In particular the induced homology homomorphism satisfies 
$$({\Phi_L})_*(\widehat{\lambda}_B^M)=0.$$
\item[(2)] Consider any $G\in \Fc(\A)$ with $\operatorname{supp}(G)=H\in \A$.\\
If $H\supseteq L$, then  $\Phi_L(\Omega^{(id_G)})=\Omega^{(id_{\pi_L(G)})}$.\\
More generally, choosing 
$R_{\pi_L(H)}:=(R_{H})_{X^G_L}/L_0$ 
for every $H\in \A$, we have 
$$({\Phi_L})_*(\wg_{H})=\left\{\begin{array}{ll}\wg_{\pi_L(H)} & \textrm{if } H\supseteq L \\
0 & \textrm{if } H \nsupseteq L
\end{array}\right.
$$
\item[(3)] For every layer $Y$ and all $F_0\in \Fc(\A_0)$ with 
$$
\Phi_L(\Sc_{Y,F_0}) \subseteq \Sc_{\pi_L(Y),\pi_L(F_0)}.
$$
\end{itemize}
\end{Lem}
\begin{proof}
(1): With Remark \ref{rem:mini}, for all $F\subseteq M$, we have 
$(C_i^F)_{\FL}=(W_i^F)_{\FL}=(R_M)_{\FL}$
for all  faces $F\subseteq M$  and $0\leq i\leq k(M,F)$. Therefore, comparing with Definition \ref{def:Lambda} and \ref{defpath}
we see that 
$\Lambda_B^M=\{(\pi_L(F),[(R_M)_{\FL},(R_M)_{\FL}])\}$, a singleton.

(2): If $H\supseteq L$, then $X^G_L=H$. Direct computations of the image under $\Phi_L$ for each of the elements of $\Omega^{id_G}$ (Remark \ref{rem:OmegaId}) and comparison with Definitions \ref{df:Om} and \ref{def:omegahat} verify the claims in this case. If $H\not\supseteq L$, then $X^G_L$ is 
the full vector space
and $\Phi_L(\Omega^{(id_G)})$ is a single vertex.

(3): Recall from \S\ref{setup6} that $\Sc_{Y,F_0}$ is defined as $\GC\mathscr D_{Y,F_0}$ for a subdiagram $\mathscr D_{Y,F_0}$ of $\mathscr D$ on the index category $\Fc(\A^Y)$. Similarly, if $\overline{\mathscr D}$ is the diagram giving $\Sc(\AL)$, then $\Sc_{\pi_L(Y),\pi_L(F_0)}$ is $\GC\overline{\mathscr D}_{\pi_L(Y),\pi_L(F_0)}$ for the subdiagram $\overline{\mathscr D}_{\pi_L(Y),\pi_L(F_0)}$ of $\overline{\mathscr D}$ over the index category $\Fc(\AL^{\pi_L{(Y)}})$. Since $\Phi_L$ is induced by the natural transformation $(\pi_L,b_{X_L}^*)$, in order to prove the claim we have to prove that $(\pi_L,b_{X_L}^*)$ restricts to a natural transformation $\mathscr D_{Y,F_0} \Rightarrow \overline{\mathscr D}_{\pi_L(Y),\pi_L(F_0)}$.

Since $\pi_L(\Fc(\A^Y)) \subseteq \Fc((\AL)^{\pi_L(Y)})$, we are left proving that, for every face $F$ in $\Fc(\A^Y)$, $b_{X^{\ast}_L}(\Sc^{F_0}(\A[F]))\subseteq \Sc^{\pi_L(F_0)}(\AL[\pi_L(F)])$.

Pick any $[G,K]\in\Sc^{F_0}(\A[F])$. By definition this means that $K=B_G$ for some chamber $B\in \Tc(\A[F])$ adjacent to $F_0$. Now, since $\pi_L^F$ is order preserving, $\pi_L^F(B)$ is adjacent to $\pi_L^F(F_0)$ and  $\pi_L^F(B)_{\pi_L^F(G)}=(B_{\FL})_{G_{\FL}} = (B_G)_{\FL}=\pi_L^F(K)$, where the second equality uses Remark \ref{rem:quisquilie}.(2). Thus, $b_{X_L^F}([G,K])\in \Sc^{\pi_L(F_0)}(\AL[\pi_L(F)])$.

\end{proof}

\begin{Cor}\label{lem:Phi_bis}

Fix a chamber $B \in \Tc(\A_0)$ and a layer $L \in \Cc(\A)$.
We have 
$$
({\Phi_{L}})_*(H_*(\Sc_{T,B};\Q)) \subset H_*(\Sc_{\pi_L(T),\pi_L(B)};\Q)
$$
In particular, for every $M\in \Cc_{d-1}$ we have 
\begin{equation}\label{eq:formula_phi}
({\Phi_L})_*(\wl_B^M) \in  H_*(\Sc_{\pi_L(T),\pi_L(B)};\Q).
\end{equation}
\end{Cor}
\begin{proof} 
The result follows from the following commutative diagram:
$$
\begin{tikzcd}
\gre{\vg_B^M} \arrow[r,hookrightarrow]&\gre{\Sc_{T,B} }
\arrow[d, "{\Phi_L}"]
\arrow[r,hookrightarrow] & 
\Sal(\A)
\arrow[d,"\Phi_L"]\\
&\gre{\Sc_{\pi_L(T),\pi_L(B)}}
\arrow[r,hookrightarrow]
& 
\Sal(\AL)
\end{tikzcd}
$$
where the existence of the leftmost vertical arrow follows from Lemma \ref{lem:Phi}.(3), and the leftmost inclusion follows from Remark \ref{rem:contenimenti}.
\end{proof}

%
%


\begin{Rem}\label{rmk:restrizioni}
Let $L \in \Cc$ be a layer and consider the map 
$\pi_L: \Cc(\A) \to \Cc(\AL)$.
Let $Y\subseteq L$.
We have that $\pi_{L}(L) = \pi_{L}(Y)$ if and only if $Y \subseteq L$. 
\end{Rem}


\begin{Rem}\label{rem:generatiindimensioneuno}
Let $F \in \Fc(\A_0)$ with $\supp(F)=L_0$. Recall from Remark \ref{htprod}
that the subcomplex $\gre{\Sc_{L,F}} \subset \Sal({\A})$ is homotopy equivalent to the product $L \times M(\A_\CCC[L])$, where $M(\A_\CCC[L])$ is the complement of the essentialization of the complexified central linear arrangement $\A_\CCC[L] = \A[L] \otimes_\R \CCC$. Hence the cohomology ring of 
$\Sc_{L,F}$
is generated in degree $1$.

In particular the cohomology ring $H^*(\Sc_{\pi_{L}(L),\pi_L(F)};\Q)$ is the Orlik-Solomon algebra generated by the restrictions of the forms $\omega_{\pi_L(H)}$ for $L \subset H$.
\end{Rem}

\subsection{Group action and inclusions}

\begin{Lem} \label{lem:traslazione}
	Let $\A$ be a toric arrangement. Assume that $\A$ is invariant by the action of an element $g \in T$. Then the multiplication by $g$ induces maps
	$\mu_g:\Sal(\A) \to \Sal(\A)$ and $\mu_g: \Sc_{L,F} \to \Sc_{L,F}$ such that the following diagram commutes
	$$
	\begin{tikzcd}
        M(\A) \ar["g"]{r} & M(\A)\\
        \Sal(\A) \ar[hook]{u} \ar["\mu_g"]{r} & \Sal(\A) \ar[hook]{u} \\
        \gre{\Sc_{L,F}} \ar[hook]{u}\ar["\mu_g"]{r} & \gre{\Sc_{L,F}} \ar[hook]{u}
	\end{tikzcd}
$$
Moreover the map $\mu_g: \gre{\Sc_{L,F} }\to \gre{\Sc_{L,F}}$ is homotopy equivalent to the identity.	
\end{Lem}
\begin{proof}
The multiplication by $g$ on $T$ lifts to a translation $\tau_g$ in the universal cover $ \mathbb R^d$ of $T$, where the periodic arrangement $\A^{\upharpoonright}$ is invariant under $\tau_g$. In particular, $\tau_g$ leaves the poset of faces invariant and, hence, induces an automorphism $\mu^\upharpoonright_g$ of the Salvetti complex $\Sal(\A^\upharpoonright)$. Now $\tau_g$ commutes with the standard inclusion $\iota: \Sal(\A^\upharpoonright)\hookrightarrow M(\A^\upharpoonright)$ as well as with the translations of $\mathbb Z^d\subseteq \mathbb R^d$. Hence, so does $\mu^\upharpoonright_g$ and, since $\Sal(\A)=\Sal(\A^\upharpoonright)/\mathbb Z^d$ (see \cite[Theorem~\olc{4.1.3}]{caldel17}), it induces the required map $\mu_g$ and the top half of the diagram commutes. 
Now, the explicit form of $\mu_g$ as a cellular map on $\Sal(\A)=\vert\int \mathscr D\vert$ is 
$$
\mu_g(F,[G,C]) = (g F, [G,C])
$$
for all $F\in \Fc(\A)$ and every $[G,C]\in \Sc(\A[F])=\Sc(\A[gF])$. (Recall that, by definition, $\A[F]=\A[L]$ where $L$ is the layer supporting $F$, and so $\A[F]=\A[gF]$.) 

In particular, this map restricts to every layer $L$ and to $\mu_g:\Sc_{L,F}\to \Sc_{L,F}$. Under the homotopy equivalence $|\Sc_{L,F}|\simeq |\Fc(\A^L)| \times \Sal(\A[L])$ of Remark \ref{htprod} 
the map $\mu_g$ is the identity on the second component and the cellular map induced by multiplication with $g$ in $T$ in the first component. But the continous map $L\to L$ defined by multiplication with $g$ is homotopic to the identity on $L\simeq  |\Fc(\A^L)|$  and any homotopy accomplishing this can be composed with the identity to give a homotopy between $\mu_g:\Sc_{L,F}\to \Sc_{L,F}$ and the identity.
\end{proof}

\begin{Lem} \label{lem:sottoarrangiamento}
 Let $\A'$ be a sub-arrangement of $\A$ of the same rank. The inclusion $M(\A) \into M(\A')$ induces a map
 $$
 \Psi: \Sc(\A) \to \Sc(\A')
 $$
 that restricts to 
 $$\Psi: \Sc_{L,F} \to \Sc_{\tilde{L},\tilde{F}}
 $$
 where $\tilde{L}$ (resp.\ $\tilde{F}$) is the smallest layer of $\A'$ containing $L$ (resp.\ the smallest face of $\A'$ with support $\tilde{L}$ containing $F$) such that the following diagram commutes up to homotopy (as before, we slightly abuse notation by using $\Phi$ for the induced cellular map $\Sal(\A)\to\Sal(\A')$).
 $$
 \begin{tikzcd}
        M(\A) \ar[hook]{r} & M(\A')\\
        \Sal(\A) \ar[hook]{u} \ar["\Psi"]{r} & \Sal(\A') \ar[hook]{u} \\
        \gre{\Sc_{L,F}} \ar[hook]{u}\ar["\Psi"]{r} & \gre{\Sc'_{\tilde{L},\tilde{F}}} \ar[hook]{u}
	\end{tikzcd}
$$
\end{Lem}
\begin{proof}
Consider the arrangements $\A^{\upharpoonright}$ and $(\A')^{\upharpoonright}$ in $\mathbb R^d$. Clearly every open cell $F$ of the polyhedral stratification of $V$ induced by $\A^{\upharpoonright}$ is contained in a unique cell $s(F)$ of the stratification induced by $(\A')^{\upharpoonright}$. This defines an order-preserving function $s: \Fc(\A^{\upharpoonright})\to \Fc ((\A')^{\upharpoonright})$ that  
induces  a poset map $\Sc(\A^{\upharpoonright})\to \Sc((\A')^{\upharpoonright})$, $[G,C]\mapsto [s(G),s(C)]$ (consider any two chambers $C$, $C'$ of $\A^\upharpoonright$: every hyperplane separating $s(C)$ from $s(C')$ also separates $C$ from $C'$, therefore $[G,C]\geq [G',C']$ implies $[s(G),s(C)]\geq [s(G'),s(C')]$).  

Recall that the inclusions $\xi:\Sal(\A^{\upharpoonright})\hookrightarrow M(\A^\upharpoonright)$, resp.\ $\xi':\Sal({\A'}^{\upharpoonright})\hookrightarrow M({\A'}^\upharpoonright)$, are defined by affine extension of any choice of a point $\xi([G,C])\in U([G,C])$, where we write $U([G,C]):=G+ iC$, for every $[G,C]\in \Sc({\A}^{\upharpoonright})$, resp.\ $\xi'([G',C'])\in U([G',C'])$ for every $[G',C']\in \Sc({\A'}^{\upharpoonright})$. In particular, different choices of such points will extend to homotopic maps.

Now consider any $x:=[G,C]\in \Sc(\A^{\upharpoonright})$. Then, $U([G,C])\subseteq U(s([G,C])$ and, by convexity, the latter contains the segment $\alpha_{x}(t)$, $0\leq t \leq 1$, between $\xi([G,C])$ and $\xi'([G,C])=\xi'[s(G),s(C)]$. Moreover, for every chain $x_1<\ldots < x_k$ in $\Sc(\A^\upharpoonright)$ and every $t$ the simplex $\sigma:=\operatorname{conv}\{\xi(\alpha_{x_i}(t))\}_i$ is contained in $M((\A')^\upharpoonright)$ (indeed $\alpha_{x_i}(t)\in U(s(x_i))$ for all $i$ and $s(x_1)<\ldots < s(x_k)$, so that $\xi'(s(x_i)):= \alpha_{x_i}(t)$ is a valid choice for the construction of $\xi'$ and $\sigma$, being in the image of the affine extension of this $\xi'$, is contained in $M((\A')^\upharpoonright)$ as explained above).
Now setting $\xi_t(x):=\alpha_{x}(t)$ and letting $\xi_t$ be the affine extension on $\Sal(\A)$ defines the desired homotopy between $\xi$ and $\xi'(s)$ in $M((\A')^\upharpoonright)$, and thus
 the following diagram commutes up to homotopy.
$$
\begin{tikzcd}
        M(\A^{\upharpoonright}) \ar[hook]{r} & M((\A')^{\upharpoonright})\\
        \Sal(\A^{\upharpoonright}) \ar[hook]{u} \ar["s"]{r} & \Sal((\A')^{\upharpoonright}) \ar[hook]{u} 
\end{tikzcd}
$$
Passing to the torus means considering the quotient by the action of the group of translations $\mathbb Z^d\subseteq V$ - call $q: V\to T$ this map. Since $s$ commutes with $q$, it descends to a cellular map 
$$\Psi: \Sal(\A) \to \Sal(\A').$$
 Since the inclusion $\xi$ (and hence $\xi'$) can be chosen equivariantly, the image under $q$ of the above diagram is a commutative diagram
$$
\begin{tikzcd}
        M(\A) \ar[hook]{r} & M(\A')\\
        \Sal(\A) \ar[hook]{u} \ar["\Psi"]{r} & \Sal(\A') \ar[hook]{u}
\end{tikzcd}
$$
Now notice that for every $F\in \Fc(\A)$ the face $\tilde{F}:=q(s(F^{\upharpoonright}))$ of $\A$ is well-defined and independent of the choice of $F^{\upharpoonright}$ in $q^{-1}(F)$. Therefore, as a cellular map between $\Sal(\A)=\vert \int \mathscr D (\A) \vert$ and $\Sal(\A')=\vert \int \mathscr D(\A') \vert$, $\Psi$ is given on vertices as 
\begin{equation}\label{eq:explPsi}\Psi(F,[G,C])=(\tilde{F},[\tilde{G},\tilde{C}]).
\end{equation}
It follows that $\Psi$  restricts to a functor 
$\Sc_{L,F}\to \Sc_{\tilde{L},\tilde{F}}.$
\end{proof}

\begin{Cor}\label{rem:sottoarrangiamento}\strut
Let $\Psi: \gre{\Sc_{L,F}} \to \gre{\Sc'_{\tilde{L},\tilde{F}}}$ be the map defined in the previous lemma.
\begin{itemize}
\item[(1)] If $H$ is such that $H\not\supset \widetilde{L}$, the image $\Psi_\ast (\wg_H)$ vanishes.
\item[(2)] If $H\supseteq \widetilde{L}$, then $\Psi_\ast (\wg_H)=\wg_H$.
\end{itemize}
\end{Cor}
\begin{proof}
In order to prove item (1), notice that we can choose a representative of the homology class $\wg_H$ to be $\Omega^{id_G}\subseteq \Sc_{L,F}$ for some face $G$ of $\A$ supported solely on $H$. 
Now, in this case $\A'[\widetilde{G}]$ is the empty arrangement with a unique face $K$. Now applying Equation \eqref{eq:explPsi} to the explicit expression of $\Omega^{id_G}$ given in Remark \ref{rem:OmegaId}, the   image $\Psi(\Omega^{id_G})$ is the single vertex $(\widetilde{G},[K,K])$.

For item (2), let $i:\A' \into \A$ be the inclusion map and assume that for $H \in \A'$ we choose $R_{i(H)} \in \Tc(\A_0)$ as the smallest chamber in $\Tc(\A_0)$ that contains $R_H \in \Tc(\A'_0)$. 
Then if $H\supseteq \widetilde{L}$
with Equation \eqref{eq:explPsi} one checks that  $\Psi(\Omega^{id_G})=\Omega^{id_{\widetilde{G}}}$, thus $\Psi_\ast (\wg_H)=\wg_H$.
\end{proof}

%
%


\section{Presentation of the cohomology ring}

\subsection{Choices for a presentation}
Let $\A$ be an essential  toric arrangement in a torus $T$ of dimension $d$.
In order to provide a presentation of the cohomology ring $H^*(M(\A);\Q)$ from the combinatorial data we need to make some choices.

\begin{choice} \label{choice_BL}
For every layer $L \in \Cc$ we choose a chamber $B(L) \in \Tc(\A_0)$ such that the intersection $\overline{B}(L) \cap X_L$ has maximal dimension  and we set $$F(L) :=\overline{B}(L) \cap X_L \in \Fc(\A_0).$$ Hence we will simply write $\Sc_{L}$ for $\Sc_{L,F(L)}$.

\end{choice}

\begin{df}\label{def_varphi_L}
For any layer $L$ let  $\incl_L: \gre{\Sc_{L,F(L)}} \to \Sal(\A)$ be the inclusion map induced by the natural transformation between the diagrams $\mathscr D_{L,F(L)}$ and $\mathscr D$.
\end{df}

Notice that the face $F(T)$ is a chamber in $\T(\A_0)$. When the setting of the toric arrangement $\A$ is understood we will write $\vg^M$ for $\vg^M_{F(T)}$ and $\wl^M$ for $\wl^M_{F(T)}$. It follows from Remark \ref{rem:contenimenti} that we have $\vg^M \subset \Sc_T= \Sc_{T,F(T)}$. Hence $\wl^M$ is a homology class in $H_1(\Sc_T;\Q)$.

\begin{choice}\label{ch:base}
Once and for all we choose elements $M_1, \ldots, M_d \in \Cc_{d-1}$ such that 
$$
\widehat{\mathcal{B}}_T(\A) := \{ \wl^{M_1}, \ldots, \wl^{M_d} \}
$$
is a basis of  $H_1(\Sc_T;\Q)$.

\end{choice}

This can be done since the arrangement is essential and hence the $\Q$-span of the defining characters is $\Hom(T,\CCC^*)\otimes \Q \simeq H^1(\Sc_T;\Q).$ 
By duality the set of $1$-dimensional layers $\Cc_{d-1}$ generate $H_1(T;\Q)$. 
Moreover we have that the projection $\gre{\Sc_T} \to T_c$, which is a homeomorphism, maps $\wl^M \mapsto M$.
Hence the set $\{\wl^M \mid M \in \Cc_{d-1} \}$ generates $H_1(\Sc_T;\Q)$. 

\begin{df} We define the following set:
	$$
	\widehat{\mathcal{B}}(\A) := 
	\{ \wl^{M_1}, \ldots, \wl^{M_d} \} \cup \{\wg_{H} \mid H \in \A\} 
	$$
\end{df}
\begin{Lem} The set $\widehat{\mathcal{B}}(\A)$ 
is a basis of $H_1(\Sal(\A);\Q)$. 
\end{Lem}
\begin{proof} Under the natural map $\Sal(\A)\to T_c$ (see \cite[Remark~\olc{4.1.6}]{caldel17}) each of the classes $\wg_H$ maps to a trivial homology class in $T_c$ (in fact, all vertices of $\gre{\Omega^{(id_G)}_H}$ map to the same vertex of the cellularization of $T_c$,  see Remark \ref{rem:OmegaId}).
Moreover, since for each $H$ we have that ${\Phi_H}_*(\wg_H)$ is a non-trivial class in $H_1(\Sal(\overline{\A}_H);\Q)$ while the cycle ${\Phi_{H'}}_*(\wg_{H})$ is trivial in $H_1(\Sal(\overline{\A}_{H'});\Q)$ for $H' \neq H$ (see Lemma \ref{lem:Phi}), we have that the classes $\wg_H$ are linearly independent. Finally, from the computation of Betti numbers of $M(\A)$ carried out, e.g., in \cite{dp2005}, we have the equality $$\rk ( H_1(\Sal(\A);\Q) )= \rk (T ) + |\A|.$$ This implies that the set $\widehat{\mathcal{B}}(\A)$ is a basis of $H_1(\Sal(\A);\Q)$. 
\end{proof}

\begin{df}

Now we can define the set of classes 
$$
\mathcal{B}_T(\A) := \{ \lambda^{M_1}(T), \ldots, \lambda^{M_d}(T)
\}
$$
as the basis of $H^1(\Sc_T;\Q)$ that is dual to $\widehat{\mathcal{B}}_T(\A)$. 

\end{df}

\begin{Rem}
	 Recall that there is a natural projection on the compact torus $\Sal(\A) \to T_c$  that induces by restriction maps $\gre{\Sc_L} \to T_c$ (see \cite[Remark~\olc{4.1.6} and Lemma \olc{4.2.11}]{caldel17}).
	 In particular there is a natural isomorphism between $H^*(\Sc_T;\Q)$ and $H^*(T_c;\Q)$ induced by the maps  $\gre{\Sc_T} \to T_c$.
	 Using the inclusions 
	 $$\gre{\Sc_T} \into \Sal(\A) \to T_c$$ we can identify $H^*(T_c;\Q) \simeq H^*(\Sc_T;\Q)$ with a sub-algebra of $H^*(\Sal(\A);\Q)$.
\end{Rem}
\begin{df}\label{firsT}
For $i = 1, \ldots, d$ we will write $\lambda^{M_i}$ for the cohomology class in $H^1(\Sal(\A);\Q)$ that is the restriction of the class $\lambda^{M_i}(T) \in    H^1(\Sc_T;\Q) \simeq H^1(T_c;\Q).$ 
Hence we can define the classes $\omega_H$, for $H \in \A$, such that the set
$$
\mathcal{B}(\A) := \{ \lambda^{M_1}, \ldots, \lambda^{M_d} \} \cup \{\omega_{H}  \mid H \in \A\}
$$
is the basis of $H^1(M(\A);\Q)$ that is dual to $\widehat{\mathcal{B}}(\A)$.


\end{df}

\begin{df}\label{df:ideale_I}
We will write $\ideal \in H^*(M(\A);\Q)$ for the ideal of $H^*(M(\A);\Q)$ generated by the classes $\lambda^{M_1}, \ldots, \lambda^{M_d}$. 
\end{df}
Notice that, since $\ideal$ is generated by the restriction of the classes in $H^1(T_c;\Q)$ to $H^*(M(\A);\Q)$, the definition of the ideal $\ideal$ depends on the choice of the chamber $B(T)$, but not on the choice of the layers $M_1, \ldots, M_d$.

\begin{choice}
For every layer $L \in \Cc$ of rank $k>0$   we choose elements  $N_1(L), \ldots$, $N_k(L) \in \Cc_{d-1}$ contained in $L$ and such that the cycles
$$
\{ \wl^{N_1(L)}_{B(L)}, \ldots, \wl^{N_k(L)}_{B(L)} \}
$$
are linearly independent in $H_1(\Sc_{L,F(L)};\Q)$.
\end{choice}
Consider now the set 
$$\A_L=\{H_{i_1}, \ldots, H_{i_s}\} \subseteq \A$$
of all hypertori $H_{i_j}$ such that $L \subset H_{i_j}$. We can generalize the notation of Choice \ref{ch:base} and extend Definition \ref{firsT} as follows.
\begin{df} Given any $L\in \Cc$, let
$$
\widehat{\mathcal{B}}_L(\A) := \{ \wl^{N_1(L)}_{B(L)}, \ldots, \wl^{N_k(L)}_{B(L)} \} \cup \{ \wg_{H_{i_1}}, \ldots, \wg_{H_{i_s}}\}.
$$
\end{df}
Clearly the set $\widehat{\mathcal{B}}_L(\A)$ is a basis of  $H_1(\Sc_{L,F(L)};\Q)$.
\begin{df} \label{df:b_cappuccio_L}
We write 
$$
\mathcal{B}_L(\A) = \{ \lambda^{N_1(L)}_{B(L)}, \ldots, \lambda^{N_k(L)}_{B(L)} \} \cup \{ \omega_{H_{i_1}}(L), \ldots, \omega_{H_{i_s}}(L)\}
$$
for the basis of  $H^1(\Sc_{L,F(L)};\Q)$ that is dual to $\widehat{\mathcal{B}}_L(\A)$.
\end{df}

\begin{df}\label{df:ideale_I_L}
	We define $\ideal(L) \subset H^*(\Sc_{L,F(L)};\Q)$ as the ideal of $H^*(\Sc_{L,F(L)};\Q)$ generated by the classes 
	$ \lambda^{N_1(L)}_{B(L)}, \ldots, \lambda^{N_k(L)}_{B(L)} $.
\end{df}

\subsection{Ideals and cohomology maps}


\begin{Lem}\label{lem:iota_cohom}

The induced restriction homomorphism in cohomology 
$
\incl_L^*:H^1(\Sal(\A);\Q) \to H^1(\Sc_{L,F(L)};\Q)
$
maps as follows:
\begin{equation}\label{eq:phi_star_lambda}
	\incl_L^*: \lambda^{M_i} \mapsto \sum a_{hi} \lambda_{B(L)}^{N_h(L)}
\end{equation}
where the coefficients $a_{hi}$ are given by the relation $\wl_{F(T)}^{N_h} = \sum a_{hi} \wl_{F(T)}^{M_i}$.

If $L \subset H$ we have
\begin{equation} \label{eq:phi_star_omega}
\incl_L^*: \omega_{H} \mapsto
\omega_{H}(L) +
\sum_{h=1}^k
\sum_{\substack{
		F\subseteq N_h(L) \\ H \in S_F(B(L),F(T))\setminus \A_{N_h(L)} 
}} \!\!\! \epsilon(H,B) \lambda_{B(L)}^{N_h(L)}; 
\end{equation}
while if $L \nsubset H$ we have
\begin{equation}\label{eq:phi_star_omega_prime}
\incl_L^*: \omega_{H} \mapsto
\sum_{h=1}^k
\sum_{\substack{
		F\subseteq N_h(L) \\ H \in S_F(B(L),F(T))\setminus \A_{N_h(L)} 
}} \!\!\! \epsilon(H,B) \lambda_{B(L)}^{N_h(L)}.
\end{equation}
\end{Lem}
\begin{proof}
In this proof we use the symbol $\int$ for the natural pairing between homology and cohomology.

Using the homological basis $\mathcal{B}_L(\A)$ of 	 $H_1(\Sc_{L,F(L)};\Q)$ we can compute
\begin{align*}
\int_{\wl_{B(L)}^{N_h}} \incl_L^*(\lambda^{M_i})& = \int_{\wl_{B(L)}^{N_h}} \lambda^{M_i} =\\
& = \int_{\wl_{F(T)}^{N_h}} \lambda^{M_i}  + \sum_{\cdots} \epsilon_H \int_{\wg_{H}}  \lambda^{M_i} = \\
& = \int_{\wl_{F(T)}^{N_h}} \lambda^{M_i} = \\
& = \int_{\sum a_{hj} \wl_{F(T)}^{M_j}} \lambda^{M_i} = a_{hi}.
\end{align*}
where the second equality follows from Equation \eqref{eq:lambda_cap}. Notice that we don't need to specify the range of the sum and the coefficients of the integrals over the terms $\wg_{H}$ since they are not significant in the computation.
In the same way we have:
\begin{align*}
\int_{\wg_H} \incl_L^*(\lambda^{M_i})& = \int_{\wg_H} \lambda^{M_i} = 0.
\end{align*}
These computations shows that pairing the two sides of Equation \eqref{eq:phi_star_lambda} against every element of a basis of 
$H_1(\Sc_{L,F(L)};\Q)$
gives the same result. Hence Equation \eqref{eq:phi_star_lambda} holds.

The analogous computation for $\omega_{H}$ gives:
\begin{align*}
\int_{\wl_{B(L)}^{N_h}} \incl_L^*(\omega_{H})& = \int_{\wl_{B(L)}^{N_h}} \omega_{H} =\\
& = \int_{\wl_{F(T)}^{N_h}} \omega_{H}  + \sum_{\substack{
		F\subseteq N_k \\ H'\in S_F(B(L),F(T))\setminus \A_{N_h}
}} \epsilon(H,B(L))\int_{\wg_{H'}}  \omega_{H} = \\
& =  \sum_{\substack{
		F\subseteq N_k \\ H'\in S_F(B(L),F(T))\setminus \A_{N_h}
}} \epsilon(H,B(L)) \delta_{H,H'}.
\end{align*}
and
\begin{align*}
\int_{\wg_{H'}} \incl_L^*(\omega_H)& = \int_{\wg_{H'}} \wg_H = \delta_{H,H'}.
\end{align*}
where $H' \in \A_{L}$ and hence there is a non-zero pairing if and only if $L \subset H$. This implies Equations \eqref{eq:phi_star_omega} and \eqref{eq:phi_star_omega_prime} via pairing of both sides with a basis of $H_1(\Sc_{L,F(L)};\Q)$.
\end{proof}

Recall Definitions \ref{df:ideale_I} and \ref{df:ideale_I_L} for $\ideal$ and $\ideal(L)$. 
The following result is a straightforward consequence of the formulas of Lemma \ref{lem:iota_cohom}.
\begin{Cor}\label{cor:phi_ideali}
The homomorphism $\incl_L^*$ maps $\incl_L^*(\ideal) \subset \ideal(L)$.
\end{Cor}

\begin{Cor}\label{cor:nullproduct}
Let $\mu \in H^s(\Sal(\A);\Q)$ be the restriction of the class $\mu_T \in H^s(T_c;\Q)$. Then $\incl_L^*(\mu)=0$ if $\rk(L)>d-s$.
\end{Cor}
\begin{proof}
This follows immediately from Lemma \ref{lem:iota_cohom} since a product of more than $d - \rk(L)$ classes of the type $\lambda_{B(L)}^M$ is zero in $H^*(\Sc_{L,F(L)};\Q)$.
\end{proof}

The results of Lemma \ref{lem:Phi} and Corollary \ref{lem:Phi_bis} have the following consequence in cohomology.

\begin{Lem}\label{lem:Phi_star}
Let $\A$ be a toric arrangement in the torus $T$. Let $L \in \Cc$ be a layer of $\A$. Consider the quotient arrangement $\AL$ in $\overline{T}=T/L_0$ and the cellular map $\Phi_{L}: \Sal(\A) \to \Sal(\AL)$. 
Moreover, assume that $ \pi_L(F(T)) = F(\overline{T})$.
Then for any hypertorus $H \in \A_L$ the cohomology homomorphism 
$$
\Phi_{L}^*: H^1( \Sal(\AL);\Q) \to H^1(\Sal(\A);\Q)
$$
induced by $\Phi_{L}$
maps as follows:
$$
\Phi_{L}^*: \omega_{\pi_L(H)} \mapsto \omega_{H}.
$$
\end{Lem}
\begin{Cor}\label{cor:1forme}
Let $\A$ be a toric arrangement in the torus $T$. Let $L \in \Cc$ be a layer of $\A$. Consider the quotient arrangement $\Sal(\AL)$ in $\overline{T}=T/L_0$ and the cellular map $\Phi_{L}: \Sal(\A) \to \Sal(\AL)$. 
Moreover assume that $\pi_L(F(T)) = F(\overline{T})$. 
Then for any layer $Y \in \Cc$, if we consider the cohomology map induced by the restriction
$${\Phi_L}_|:\Sc_{Y,F(Y)} \mapsto \Sc_{\pi_{L}(Y),\pi_L(F(Y))}$$ the following holds:
for any hypertorus $H \in \A$ such that $Y \subset H$ and $L \subset H$ we have
$$
{{\Phi_L}_|}^*(\omega_{\pi_L(H)}(\pi_L(Y)) = \omega_H(Y).
$$
\end{Cor}
\begin{proof}
This follows immediately from the previous lemma using the commutativity of the diagram
$$
\begin{tikzcd}
\gre{\Sc_{Y,F(Y)} }
\arrow[d, "{\Phi_L}"]
\arrow[r,hookrightarrow, "\incl_Y"] & 
\Sal(\A)
\arrow[d,"\Phi_L"]\\
\gre{\Sc_{\pi_L(Y),\pi_L(F(Y))}}
\arrow[r,hookrightarrow, "\incl_{\pi_L(Y)}"]
& 
\Sal(\AL).
\end{tikzcd}
$$
\vspace{-2em}

\end{proof}
We can assume the arrangement $\A$ to be totally ordered; let us fix such an ordering.  Moreover, recall that, given any $(d-r)$-dimensional layer $L \in \Cc_r$ in the torus $T$ of dimension $d$, we can consider the linear arrangement $\A[L]$. 

\begin{df}
The ordering of $\A$ induces an ordering of $\A[L]$. For every element of the no-broken-circuit-basis (or nbc-basis, see \cite{yuz2001}) associated to the arrangement $\A[L]$ we can consider the corresponding ordered subset $S = (H_{i_1}, \ldots, H_{i_r}) \subset \A$.

We will write $\omega_{S}$ for the product $\omega_{H_{i_1}} \cdots \omega_{H_{i_r}} \in H^r(M(\A);\Q)$.
\end{df} 

\begin{Cor}\label{cor:Phi_star_alpha}
	Let $\A$ be a toric arrangement in the torus $T$ of dimension $d$. Let $L \in \Cc$ be a layer of $\A$ of rank $r$. 
	Consider the quotient arrangement $\AL$ in $\overline{T}=T/L_0$ and the cellular map $\Phi_{L}: \Sal(\A) \to \Sal(\AL)$. 
	Assume that $\pi_L(F(T)) = F(\overline{T})$. 
	
	Let $S = (H_{i_1}, \ldots, H_{i_r})$ be an element of the nbc-basis of $\A[L]$ and let $\alpha \in H^*(\Sal(\AL);\Q)$ be a class  such that $\alpha - \omega_{\pi_L(S)}$ restricts to zero in $H^*(\Sc_{\pi_L(L)};\Q)$ and $\alpha$ restricts to a class in $\ideal(\overline{L}') \subset H^*(\Sc_{\overline{L}'};\Q)$ for $\overline{L}' \neq \pi_L(L) $.
	Then  we have that:
	\begin{enumerate}
	\item[a)] 
	if $\pi_L(L') = \pi_L(L)$ (that is, if $L' \subseteq L$)
	the class 
	$({\Phi_{L}}^*(\alpha) - \omega_S)$ restricts to $0$ in $H^*(\Sc_{L'};\Q)$;
	\item[b)]
	if $\pi_L(L') \neq \pi_L(L)$
	(that is, if $L' \nsubseteq L$)
	the class
	${\Phi_{L}}^*(\alpha)$ restricts to  a class of $\ideal(L')$ in $H^*(\Sc_{L'};\Q)$.
	\end{enumerate}

\end{Cor}
\begin{proof}
The result follows from Corollary \ref{cor:1forme} by multiplicativity.
\end{proof}

\begin{Rem}\label{rem:tor_free}
The cohomology ring $H^*(\Sal(\A);\Z)$ is a free $\Z$ module (this has been proved in \cite{dd2}, but follows also from \cite[Remark~\olc{6.1.1}]{caldel17}).
\end{Rem}


\begin{Lem} \label{lem:cohom_modulo}The cohomology ring $H^*(\Sc_{L};\Z)$ is a module over $H^*(T;\Z)$ generated by the restriction of the classes $\omega_S$ for $S \in \mathrm{nbc}(\A[L])$. 
\end{Lem}
\begin{proof}
The lemma follows from \cite[Lemma \olc{4.2.15}]{caldel17},
where the homotopy equivalence
$$
\Theta: \gre{\Sc_{L,F(L)}} \to |\mathcal F(\A^L)| \times \Sal(\A[L])
$$
is given.
Since the projection on the first component of $\Theta$ is the projection $ \gre{\Sc_{L,F(L)}} \to L_c$ induced by $\Sal(\A) \to T_c$, we have that the cohomology of $\Sc_{L,F}$ is a $H^*(T;\Z)$-module generated by generators of the Orlik-Solomon algebra $H^*( \Sal(\A[L]);\Z)$. 
Finally the following claims hold:
\begin{itemize}[leftmargin=1.5em]
	\item[a)] The homology classes $\wg_{H}$ are non-trivial (in fact their image in $\Sal(\overline{\A}_H)$ is non-zero).
	\item[b)] For $H \supset L$ the homology classes $\wg_{H}$ have a representative in $\Sc_{L,F(L)}$.\\ In fact, for every $G\in \mathcal F(A^L)$ the arrangement $\A[G]$ contains $H$ and hence in particular a face $W$ supported on $H$. Then we can consider the morphism $m$ of $\mathcal F(\A)$ that arises from the order relation $G\leq W$ in $\mathcal F(\A[G])$ (ensuring that $F_m=W$). Thus we see that the poset $\Sc^{F}(\A[G])$ contains the subposet $\Omega^{m}$ of Remark \ref{rem:OmegaId}. 
	\item[c)] The classes $\wg_{H}$ project to trivial classes in $T_c=\gre{\Sc_{T}}$.
	\item[d)] If $\B$ is the essentialization of $\A \setminus \{H\}$, the classes $\wg_{H}$ maps to trivial classes in $\Sal(\B)$ (this claim follows directly from  Lemma \ref{lem:sottoarrangiamento} and Remark \ref{lem:sottoarrangiamento} when $\A$ and $\A \setminus \{H\}$ have the same rank, otherwise from Lemma \ref{lem:Phi}).
\end{itemize} 
Hence we have that the classes $\omega_{H}(L)$ for $H \supset L$, that are the restrictions of the corresponding classes $\wg_{H}$ to $\Sc_{L}$, are the standard generators of the Orlik-Solomon algebra $H^*( \Sal(\A[L]);\Z)$.
\end{proof}
\subsection{The cohomology ring}
\begin{df}
Let $\A$ be a toric arrangement in a torus $T$. Let $L \in \Cc_r(\A)$ be a layer and let $S \in \mathrm{nbc}(\A[L])$.
We define the subarrangement  $\A_S \subset \A$ as the set of hypertori of $\A$ associated to the elements of $S$.
\end{df}

\begin{df}
Let $\A$ be a toric arrangement in a torus $T$. We define the \emph{stabilizer} of $\A$ as the group $G \subset T$ given by
$$
G:=\{g \in T \mid \forall H \in \A, gH = H\}.
$$
and the \emph{essential stabilizer} of $\A$	as the group $\G$ of the connected components of $G$.
Hence $\G :=G/G_0$, where $G_0$ is the connected component of the identity of $G$.

Given a layer $L \in \Cc_r(\A)$ and an element $S \in \mathrm{nbc}(\A[L])$ we write $G_S$ (resp. $\G_S$) for the stabilizer (resp.  essential stabilizer) of $\A_S$.
\end{df}
\begin{Rem}
If $\A$ is essential then $G$ is discrete and we have $\G \simeq G$.
In general, by choosing a direct summand of $G_0$ in $T$ we get a lifting $\G \to G$. Hence we can always identify $\G$ with a subgroup of $T$ that acts by multiplication on the layers in $\Cc(\A)$.
\end{Rem}


\begin{Prop}\label{prop:def_omega_S_L}
	Let $\A$ be an essential arrangement in a torus $T$ of dimension $d$. Let $L \in \Cc_r$ a layer of rank $r \geq 0$ and let $S \in \mathrm{nbc}(\A[L])$ be a no-broken-circuit of length $r$.
	For a layer of $Y \in \Cc(\A)$ let $\overline{Y}$ be the smallest layer in $\Cc(\A_S)$ containing $Y$.
	There exists a 
	unique 
	cohomology class $\omega_{S,L} \in H^r(\Sal(\A);\Q)$ such that for every layer $Y \in \Cc(\A)$ and for every face $F \in \Fc(\A_0)$ with $\supp(F) = Y$ we have: 
	\begin{enumerate}
	\item[i)] if $L \subset \overline{Y}$ then
	 $\omega_{S,L}-\frac{\omega_{S}}{|\Stab_{\G_S}( \overline{Y})|}$
	 restricts to $0$ in $H^r(\Sc_{Y,F};\Q)$;	
	\item[ii)] if $L \nsubseteq \overline{Y}$ then $\omega_{S,L}$ restricts to $0$ in $H^r(\Sc_{Y,F};\Q)$.
		\end{enumerate}
\end{Prop}

\begin{proof} We prove our statement by induction on $d$ and $r$, considering the following cases.
	\begin{itemize}[leftmargin=1.5em]
		\item[a)]{[$d=0$]} For $L=T$ and $S = \emptyset$ we can just set $\omega_{S,L}$ as the constant class $1 \in H^0(\Sal(\A);\Q)$.
		\item[b)]{[$d=1$]} In the case where $r=0$ the set $S$ is again empty and $\omega_{S,L}=\omega_{\emptyset,T}$ is the generator $1 \in H^0(\Sal(\A);\Q)$.
		When $r=1$ the result is trivial as we can set $\omega_{S,L} = \omega_{\{H\},H} = \omega_{H}$.		
		
		\item[c)]{[$d>1, r<d$]}
		Suppose now that $\dim T = d >1$.
		We can assume that the statement is true for any essential arrangement in a torus $T'$ with $\dim T' <d$. We show that if $\rk(L)= r < d$ the statement follows also for $\dim T = d$.
		In fact, given $S =  (H_{i_1}, \ldots, H_{i_r}) \in \mathrm{nbc}(\A[L])$ of length $r$, we can consider the quotient arrangement $\AS:= \A_S/L_0$, the layer $\overline{L} := L/L_0 \in \Cc_r(\AS)$ and 
		the element $\overline{S} :=\pi_L(S)\in \mathrm{nbc}(\AS[\overline{L}])$. By induction on the dimension $d$
		the class $\omega_{\overline{S},\overline{L}}$ exists and we can set
	 $${\omega}_{S,L} := \Phi_L^*(\omega_{\overline{S},\overline{L}})$$ 
	 where $\Phi_L^*$ is the map defined in Lemma \ref{lem:Phi_star}.
	 Now by induction on the dimension $d$ the two required relation can be computed: if $\overline{L} \subset \pi_L(Y)$ (that is if $L \subset \overline{Y}$) then $\omega_{\overline{S},\overline{L}} - \frac{\omega_{\overline{S}}}  {|\Stab_{\G_{\overline{S}}}(\pi_L(Y))|}$ restricts to $0$ in $\Sc_{\pi_L(Y),R}$ and if $\overline{L} \nsubseteq \pi_L(Y)$ (that is if $L \nsubseteq \overline{Y}$) then $\omega_{\overline{S},\overline{L}}$ restricts to $0$ in $\Sc_{\pi_L(Y),R}$. Hence, since $\G_S \simeq \G_{\overline{S}}$,
	 according to Corollary \ref{cor:Phi_star_alpha} we have that $${\omega}_{S,L} - \frac{\omega_{S}}{|\Stab_{\G_S}( \overline{Y})|}$$ restricts to $0$ in $H^*(\Sc_{Y,F};\Q)$ for all $Y,F$ such that $L \subseteq \overline{Y}$ and ${\omega}_{S,L}$  restricts to $0$  in $H^*(\Sc_{Y',F'};\Q)$ if $L \nsubseteq \overline{Y'}$.
		
	\item[d)]{[$d>1, r=d$]} 
	First we consider the case when $\A=\A_S$. In this case the arrangement $\A_0$ is boolean and for any choice of $B_0\in \Tc(\A_0)$ the closure of $B_0$ meets every hyperplane in maximal dimension.
	Then 
	we can set $F(Y) = \overline{B_0} \cap Y_0$ and Theorem 4.2.17 of \cite{caldel17} holds with $F_0=F(Y)$ (see \cite[Remark 2.1]{erratum_corto} and Theorem \ref{thm:coherent} for a self-contained proof). 
As a consequence Theorem A and B of \cite{caldel17} can be applied in this setting.




From \cite[Thm.~{A}]{caldel17} we can choose a class $\omega_{S,L} \in H^r(\Sal(\A);\Q)$ such that $\omega_{S,L} - \omega_S$ restricts to zero in 
$H^*(\Sc_{L,F(L)};\Q) \simeq H^*(L)\otimes H^*{\Sal(\A[L])}$
 and $\omega_{S,L}$ restricts to zero in 
 $H^*(\Sc_{Y,F(Y)};\Q) \simeq H^*(Y)\otimes H^*{\Sal(\A[Y])}$
  for $Y \neq L$. From the definition of coherent elements (see \cite[Def.~\olc{2.3.4}]{caldel17}) and from \cite[Thm.~\olc{B}]{caldel17} we have that $\omega_{S,L}$ restricts to zero in $H^r(\Sc_{Y,F(Y)};\Q)$ for $Y \neq L$. Hence $\omega_{S,L}$ satisfies the required conditions i) and ii) when $F=F(Y)$. 

Recall that for $g\in \Stab_{\G_S}(Y)$ we have a multiplication map $g: Y \to Y$  that by Lemma \ref{lem:traslazione} is homotopic to the identity map and hence induces the identity map on $H^*(\Sc_{Y,F})$. 

We need to prove that i) and ii) are satisfied  also for every complex $\Sc_{Y,F}$, with $F \neq F(Y)$. This is trivially true for $Y=L$, since in this case $F=F(Y)$ is the unique  face in $\Fc(\A_0)$ with $\supp(F)=Y_0$. When $Y \neq L$ we can consider the class
$$
\omega^Y_{S,L}:=\sum_{g \in \Stab_{\G_S}(Y)} g^* \omega_{S,L} \in H^*(\Sal(\A);\Q).
$$
Let $S = S_1 \sqcup S_2$, where the elements of $S_1$ are the elements of $S$ that contain $Y$. Let $\epsilon \in \{\pm1 \}$ be the sign such that $\omega_S = \epsilon \omega_{S_1} \omega_{S_2}$. 
Consider the map $\Phi_{S_1}: \Sal(\A) \to \Sal(\overline{\A}_{S_1})$. 
We claim that
\begin{equation}\label{eq:spezzo}
\omega^Y_{S,L} = \epsilon \Phi_{S_1}^*(\omega_{\overline{S_1},\pi_Y(Y)}) \omega_{S_2}.
\end{equation}
The equality follows checking that the two terms agree when restricted to $\Sc_{W,F(W)}$ for every $W \in \Cc(A_S)$. In fact we have:
\begin{equation*}
\varphi_W^* (\omega^Y_{S,L})=\frac{\omega_S}{|\Stab_{\G_S}(W)|} \cdot |\Stab_{\G_S}(Y) \cap \Stab_{\G_S}(W)|
\end{equation*}
if $L \subseteq \Stab_{\G_S}(Y).W$ and
\begin{equation*}
\varphi_W^* (\omega^Y_{S,L})=0
\end{equation*}
otherwise.
On the other side we have 
\begin{equation*}
\varphi_W^*( \epsilon  \Phi_{S_1}^*(\omega_{\overline{S_1},\pi_Y(Y)}) \omega_{S_2}) = \epsilon \frac{\omega_{S_1}}{|\Stab_{\G_{\overline{S_1}}}(\pi_Y(W))|} \omega_{S_2} = \frac{\omega_{S}}{|\Stab_{\G_{\overline{S_1}}}(\pi_Y(W))|}
\end{equation*}
if $\pi_Y(Y) \subseteq \pi_Y(W)$ and
\begin{equation*}
\varphi_W^*( \epsilon  \Phi_{S_1}^*(\omega_{\overline{S_1},\pi_Y(Y)}) \omega_{S_2}) = 0
\end{equation*}
otherwise.
The equality $$\frac{ |\Stab_{\G_S}(Y) \cap \Stab_{\G_S}(W)|}{|\Stab_{\G_S}(W)|} = \frac{1}{|\Stab_{\G_{\overline{S_1}}}(\pi_Y(W))|}$$
follows 
since the kernel of the surjective homomorphism
$$
\Stab_{\G_S}(W) \to \Stab_{\G_{\overline{S_1}}}(\pi_Y(W))
$$
is exactly the subgroup $\Stab_{\G_S}(Y) \cap \Stab_{\G_S}(W)$.
Moreover the two conditions $L \subseteq \Stab_{\G_S}(Y).W$ and $\pi_Y(Y) \subseteq \pi_Y(W)$ are equivalent. In fact $\pi_Y(L) = \pi_Y(Y)$ and $\pi_Y(\Stab_{\G_S}(Y).W) = \pi_Y(W)$ and this proves that the first condition implies the second one. Since $\pi_Y(W)$ is the quotient by $Y_0$ of the smallest layer $\overline{W}$ of $\A_{S_1}$ containing $W$, the second condition means that $Y \subset \overline{W}$ and then $L$, which is a point of $Y$, is contained in $\Stab_{\G_S}(Y).W$, hence the second condition implies the first one. Then we have proved Equation \eqref{eq:spezzo}.
In particular we have by induction that for every face $F \in \Fc(\A_0)$ such that $\supp(F)=Y_0$ the class
$\omega_{S,L}$ restricts in $H^*(\Sc_{Y,F})$ as 
$\frac{\omega_S}{|\Stab_{\G_S}(Y)|}$ and this proves the proposition if $\A=\A_S$. 

If $\A \neq \A_S$ we can consider the map $\Sal(\A)\to \Sal(\A_S)$ induced by the inclusion $M(\A) \into M(\A_S)$ and the result follows applying Lemma \ref{lem:sottoarrangiamento} and Remark \ref{rem:sottoarrangiamento}.
\end{itemize}
\end{proof}

\begin{Thm}\label{thm:iniettivo}
Let $\A$ be an essential toric arrangement in $T$. 
The homomorphism of algebras
$$
\bigoplus_{L \in \Cc(\A)} \incl_L: H^*(\Sal(\A);\Z) \to \bigoplus_{L \in \Cc(\A)} H^*(\Sc_{L};\Z)
$$
is injective. 
\end{Thm}

\begin{proof}
	Let $R$ be a ring and let $\II_R$ (resp.~$\JJ_R$) be the ideal of $A_R^*:=H^*(\Sal(\A);R)$ (resp.~ $B_R^*:= \bigoplus_{L \in \Cc(\A)} H^*(\Sc_{L};R)$) generated by the restriction of $H^1(T;R) \simeq H^1(\Sc_T;R)$. Note that $\II_R$ and $\JJ_R$ are graded ideals with respect to the cohomological graduation and we will write $(\II_R)_j$ (resp.~$(\JJ_R)_j$) for the graded component of $\II_R$ (rsp.~$\JJ_R$) in $A^j_R$ (resp.~$B^j_R$). 
	
	The powers of the ideals $\II_R$  and $\JJ_R$ induce filtrations on the rings $A_R^*$ and $B_R^*$. 	
	Let $\Gr(A_R)$ and $\Gr(B_R)$ be the associated bi-graded groups, where we write $\Gr_i(A^j_R)$ (resp.~$\Gr_i(B^j_R)$) for $(\II_R^i/\II_R^{i+1})_j$ (resp.~$(\JJ_R^i/\JJ_R^{i+1})_j$).
	
	The map $\Phi:= \oplus_{L \in \Cc(\A)} \incl_L$ induces an homomorphism of bi-graded groups
	$$
	\overline{\Phi}: \Gr_i(A^j_R) \to \Gr_i(B^j_R).
	$$
	As recalled in Remark \ref{rem:tor_free}, the cohomology groups $H^*(\Sal(\A);\Z)$ and $H^*(\Sc_{L};\Z)$ are torsion free 
	and hence $A_\Z$ (resp.~$B_\Z$) includes in $A_\Q$ (resp.~$B_\Q$). Moreover the injectivity of $\overline{\Phi}$ implies the injectivity of the map $\Phi$. As a consequence of these two facts, in order to prove that $\Phi$ is injective for $R=\Z$ it will be enough to prove that $\overline{\Phi}$ is injective  when $R=\Q$. 
	
	This can be seen showing that $\Gr_i(A^j_\Q)$ and $\overline{\Phi}(\Gr_i(A^j_\Q))$ have the same dimension. In fact if we fix $L \in \Cc(\A)$ with $\rk (L) = l$ we have that $\Gr_i(H^{l+i}(\Sc_L;\Q)) \simeq H^{l}(M(\A[L]);\Q) \otimes H^i(L;\Q)$. For a given  $S \in \mathrm{nbc}(\A[L])$ with $|S|=l$ and $\lambda \in H^i(T;\Q)$, the class $\omega_{S,L} \cdot \lambda$ belongs to $\Gr_i(A^{l+i}_\Q)$. 
	
	It follows from Proposition \ref{prop:def_omega_S_L} that $\omega_{S,L}$ maps to $\omega_S(L) \otimes \lambda$ in the graded piece $\Gr_i(H^{l+i}(\Sc_L;\Q))$. 
	
	Moreover for $L' \neq L$ with $\rk(L')=l'$ we have that $\omega_{S,L}$ maps to $0$ in the graded piece $\Gr_i(H^{l+i}(\Sc_{L'};\Q))$. Again this follows from Lemma \ref{lem:iota_cohom}
	and Proposition \ref{prop:def_omega_S_L} since: either at least one of the hypertori $H_s$ for $s \in S$ does not contains $L'$ and then $\omega_{S,L}$ maps to $\JJ_R$, either $L'\subsetneq L$ and hence $\omega_S(L)$ has dimension less then $l'$ in $H^{*}(\Sc_{L'};\Q)$. 
	
	As a consequence the images of the classes $\omega_{S,L} \cdot \lambda$ for $L \in \Cc(\A)$, $S \in \mathrm{nbc}(\A[L])$ with $|S|=\rk(L)$ and $\lambda$ in a basis of $H^i(L;\Q)$ in $\Gr(B_\Q)$ are linearly independent
	and the rank of the image of $\Gr(A_\Q)$ in $\Gr(B_\Q)$ is greater or equal to 
	$$
	\sum_{L \in \Cc(\A)} 2^{d-\rk(L)} \dim H^{\rk(L)}(M(\A[L]);\Q)
	$$
	that is the dimension of $A_\Q$ (see \cite{looi93,dp2005}). Hence $\overline{\Phi}$ is an injective homomorphism.
\end{proof}

\begin{Prop}\label{prop:classi_intere}
The classes $\omega_{S,L}$ defined in Proposition \ref{prop:def_omega_S_L} are integer classes.
\end{Prop}
\begin{proof}
Following the same pattern of the proof of Proposition \ref{prop:def_omega_S_L}, we can prove the result by induction.
The claim follows immediately for the cases a) (since $1$ is an integer class) b) (since the classes $\omega_{H}$ are integer classes) and c) since the pull back of an integer class is an integer class.

Concerning case d), since the restriction of an integer class is an integer class, we can assume that $\A = \A_S$.
We will show that for a given layer $Y$ such that $L\subset Y$ and $|\Stab_{\G_S}(Y)| \neq 1$ the restriction of the class $\omega_{S,L}$ in $H^*(\Sc_{Y,F};\Q)$ is an integer class.
We consider the group $\NN:=\Stab_{\G_S}(Y)$ and the quotient $\Sal(\A) \to \Sal(\A)/\NN$.
We have the commutative diagram
$$
\begin{tikzcd}
\Sal(\A) \arrow[r]&\Sal(\A)/{\NN}\\
\gre{\Sc_{Y,F}} \arrow[u, hookrightarrow]\arrow[r, "\pi_{\NN}"] & \gre{\Sc_{Y,F}}/{\NN}\arrow[u,hookrightarrow]
\end{tikzcd}
$$
and the homotopy equivalences
$$
\begin{tikzcd}
\gre{\Sc_{Y,F}} \arrow[d,"\simeq"]\arrow[r, "\pi_{\NN}"] & \gre{\Sc_{Y,F}}/{\NN} \arrow[d,"\simeq"]\\
\Sal(\A[Y]) \times Y \arrow[r] & \Sal(\A[Y]) \times (Y/{\NN})
\end{tikzcd}
$$
Let $m= |\Stab_{\G_S}(Y)|$.
Notice that when $i=\dim(Y)$ the cohomology map $$H^i(Y/\NN;\Q) \to H^i(Y;\Q)$$ induced by the $m$-fold covering $Y \to Y/\NN$ is the multiplication by $m$.

Recall that we are assuming that $r=\rk (T) = \rk( \A) = |S|$ and hence $i  = \dim Y = r = \rk(Y)$.

The class $\omega_S \in H^r(\Sal(\A);\Z)$ is $\NN$-invariant and hence it is the pullback of an integer class $\overline{\omega}_{S} \in H^r(\Sal(\A)/\NN;\Z)$. Moreover the class $\overline{\omega}_{S}$ restricts to a class  $\beta \in H^r(\Sc_{Y,F}/\NN;\Z) \simeq H^{\rk(Y)}(\Sal(\A[Y]);\Z) \otimes H^i(Y/\NN;\Z)$. This implies that $\varphi_{Y,F}^*(\omega_{S}) = \pi_\NN^*(\beta)$ is $m$ times an integer class in $H^r(\Sc_{Y,F};\Z)$ and hence $\omega_{S,L}$ restricts to an integer class in $\Sc_{Y,F}$. From the injectivity of the map $$H^*(\Sal(\A);\Z) \to \bigoplus_{Y \in \Cc} H^*(\Sc_Y;\Z)$$
it follows that the class $\omega_{S,L}$ is an integer class.
\end{proof}


\begin{Thm}\label{thm:generano}
Let $\A$ be an essential toric arrangement in $T$. 
The integer cohomology ring $H^*(\Sal(\A);\Z)$ is generated as a module over $H^*(T;\Z)$ by the classes
$\omega_{S,L}$ for $L \in \Cc(\A)$ and $S \in \mathrm{nbc}(\A[L])$. 
\end{Thm}
\begin{proof}
Since we have proved in Theorem \ref{thm:iniettivo} that the map $\Phi = \oplus_{L \in \Cc(\A)} \varphi_L$ is injective over $\Z$, it will be enough to show that the image of the homomorphism of algebras
$$
\bigoplus_{L \in \Cc(\A)} \incl_L: H^*(\Sal(\A);\Z) \to \bigoplus_{L \in \Cc(\A)} H^*(\Sc_{L};\Z)
$$
is the $H^*(T;\Z)$-module generated by the restrictions of the classes $\omega_{S,L}$ for $L \in \Cc(\A)$ and $S \in \mathrm{nbc}(\A[L])$.

We keep the notation of the proof of Theorem \ref{thm:iniettivo}. Let $A'_R$ the sub-$H^*(T;R)$-module of $A_R$ generated by the classes $\omega_{S,L}$ for $L \in \Cc(\A)$ and $S \in \mathrm{nbc}(\A[L])$. Using the injectivity of $\Phi:A^*_\Z \to B^*_\Z$ and the fact that the surjectivity of $\overline{\Phi}:\Gr(A'_R) \to \overline{\Phi}(\Gr(A_R))$ implies the surjectivity of $\Phi:A'_R \to \Phi(A_R)$, in order to prove that $A'_\Z = A_\Z$ we will show that 
$$
\overline{\Phi}(\Gr(A'_\Z)) = \overline{\Phi}(\Gr(A_\Q)) \cap \Gr(B_\Z).
$$
In fact, since we have the inclusions $$\overline{\Phi}(\Gr(A'_\Z)) \subset \overline{\Phi}(\Gr(A_\Z)) \subset \overline{\Phi}(\Gr(A_\Q)) \cap \Gr(B_\Z),$$ the equality between the first and the last term implies the equality between the first and the second one.
As we have seen in the proof of Theorem \ref{thm:iniettivo}, 
$\overline{\Phi}$ maps the class of $\omega_{S,L}$ in $\Gr_0(H^{\rk(L)}(\Sal(\A);\Z))$ 
to the class $\omega_S(L)$ in $\Gr_0(H^{\rk(L)}(\Sc_{L};\Z))$. 
Since $H^*(\Sc_{L};\Z) \simeq H^*(M(\A[L]);\Z)\otimes H^*(L)$ 
we have that the set of classes $\omega_S(L)$ for $S \in \mathrm{nbc}(\A[L])$ is a set of generators of 
$$
\bigoplus_i \Gr_i(H^{\rk(L)+i}(\Sc_{L};\Z)) \simeq \bigoplus_i H^{\rk(L)}(M(\A[L]);\Z)\otimes H^i(L)
$$
as a $H^*(T;\Z)$-module. The sum of these modules, for $L \in \Cc(\A)$, is the intersection $\overline{\Phi}(\Gr(A_\Q)) \cap \Gr(B_\Z)$. Hence the claim follows.
\end{proof}

\begin{es}\label{es:es}
	As an example of our result we provide an explicit description of the cohomology of the complement of the toric arrangement $\A = \{H_0, H_1, H_2\}$ in $T=(\CCC^*)^2$ given by:
	$$
	H_0=\{z \in T | z_0=1\}; \quad H_1=\{z \in T | z_0z_1^2=1\}; \quad H_2=\{z \in T | z_1=1\}.
	$$
	The associated hyperplane arrangement $\A_0$ in $V=\R^2$ is given by the corresponding hyperplanes
	$$
	W_0=\{x \in V | x_0=0\}; \quad W_1=\{x \in V | x_0+ 2 x_1=0\}; \quad W_2=\{x \in V | x_1=0\}.
	$$
	We consider in $\Tc(\A_0)$ the chambers $B_0 = \{x \in V | x_0<0, x_0+2x_1>0 \}$ and $B_1=\{x \in V | x_1>0, x_0+2x_1<0 \}$ (see Figure \ref{fig:es}).
\begin{figure}[htb]
	\begin{adjustbox}{scale=0.9,center}
	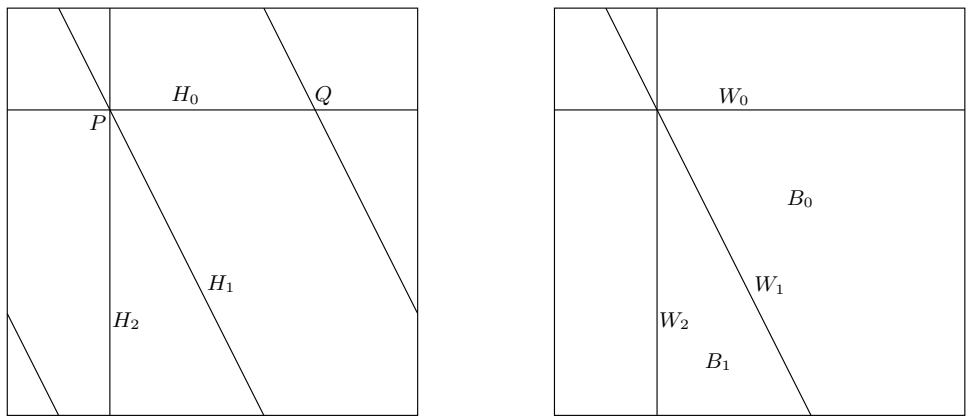	
\end{adjustbox}
\caption{The real picture of the toric arrangement $\A$ of Example \ref{es:es} (on the left) and of the corresponding central arrangement $\A_0$ (on the right).}\label{fig:es}
\end{figure}
The poset of layers $\Cc(\A)$ is given by the elements $T, H_0, H_1, H_2$ and the points $P=\{(1,1)\}$, $Q=\{(1,-1)\}$.
In order to define the subcomplexes $\gre{\Sc_L}$ for $L \in \Cc(\A)$ we need to choose the chamber $B(L)$. We can do this as follows:
$B(H_2) = B_1$, $B(L)=B_0$ for $L \neq H_2$. Moreover for $H \in\A$ choose ${}^HC=B(H)$.
As a basis $\widehat{\mathcal{B}}_T(\A)$ we can choose the set $\{\wl_{B_0}^{H_0}, \wl_{B_0}^{H_2}\}$. All the other choices of basis are natural.
In Table \ref{tab:tab1} we describe  the restriction of each generator of the cohomology of $\Sal(\A)$ to each one of the subcomplexes $\gre{\Sc_L} = \gre{\Sc_{L,F(L)}}$, for $L \in \Cc(\A)$. Cells are empty when a class restricts to zero. The multiplicative structure of the cohomology of $\Sal(\A)$ is induced by the multiplcative structure on each subcompex.
\begin{table}[ht]
\begin{center}
\begin{tabular}{ |C|C|C|C|C|C|C| }
\hline
\Sal(\A) & \Sc_T & \Sc_{H_0} & \Sc_{H_1} &\Sc_{H_2} &\Sc_{P} &\Sc_{Q} \\ 
\hline
\lambda^{H_0}_{B_0} &\lambda^{H_0}_{B_0}&\lambda^{H_0}_{B_0}&\lambda^{H_1}_{B_0}&&&\\
\hline
\lambda^{H_2}_{B_0} &\lambda^{H_2}_{B_0}&&2\lambda^{H_1}_{B_0}&\lambda^{H_2}_{B_1}&&\\
\hline
\omega_{H_0}&&\omega_{H_0}&&&\omega_{H_0}&\omega_{H_0}\\
\hline
\omega_{H_1}&&&\omega_{H_1}&2\lambda^{H_2}_{B_1}&\omega_{H_1}&\omega_{H_1}\\
\hline
\omega_{H_2}&&&&\omega_{H_2}&\omega_{H_2}&\\
\hline
\omega_{\{0,2\},P}&&&&&\omega_{H_0}\omega_{H_2}&\\
\hline
\omega_{\{1,2\},P}&&&&\lambda^{H_2}_{B_1}\omega_{H_2}&\omega_{H_1}\omega_{H_2}&\\
\hline
\omega_{\{1,2\},Q}&&&&&&\omega_{H_1}\omega_{H_2}\\
\hline
\end{tabular}
\end{center}
\caption{}\label{tab:tab1}
\end{table}

\end{es}

\newpage

\appendix \section{}

The following result is a special case of Theorem 4.2.17 of \cite{caldel17}. Notice that Theorem 4.2.17 of \cite{caldel17} does not hold in general (see Remark 2.1 in \cite{erratum_corto}).


\begin{Thm} \label{thm:coherent}
	Suppose that there is a chamber $B$ of $\A_0$ such that for every layer $L\in \Cc$ the face $F_0:=X_L\cap\overline{B}$ is full-dimensional in $X_L$. 
	For every $L\in \Cc$ let $B:=B(L)$.
	
	Fix an integer $q$ and let $L$ be a layer with $\rk(L) > q$. Consider the set 
	$(\Cc_{\leq L})_q$  of all the layers $L'$ such that 
	$L \subseteq L'$ and
	$q = \operatorname{rk}(L')$. 
	The following diagram of 
	groups is commutative.
	$$
	\xymatrix{
		H^*(\Sal(\A);\Z)\ar[drr]_{\varphi_L^*} \ar[rr]^(.35){\underset{L' \in (\Cc_{\leq L})_q}{\bigoplus}\!\!\!\!\!\varphi_{L'}^*}
		&& \underset{L' \in (\Cc_{\leq L})_q}{\bigoplus}
		\!\!\!H^q(M(\Ar{L'});\Z) \otimes H^*(L';\Z) 
		\ar[d]
		\\
		&& 
		H^q(M(\Ar{L});\Z) \otimes H^*(L;\Z)
	}
	$$
	Here the maps $\varphi_L$ and $\varphi_{L'}$ are defined in Definition \ref{def_varphi_L} and the vertical map is naturally induced by the inclusion $L\subseteq L'$ and the Brieskorn Lemma \cite[Proposition 3.3.3]{caldel17} for the hyperplane arrangements $\A[L']\subseteq \A[L]$.
	\end{Thm} 

\begin{proof}
 Let $L$ and $L'$ be as in the claim and write $F_0:=\overline{B}\cap X_L$, resp.\ $F_0':=\overline{B}\cap X_{L'}$. The assumption on $B$ ensures that $F_0 \leq F_0'$. This implies that $\Sc^{F_0'}(\A[F])$ is a subcategory of $\Sc^{F_0}(\A[F])$ for all $F\in \Fc(\A^L)$ -- see \eqref{def_subposets_local} for definitions.   Now we can consider $i_{F_0'}:\Fc(\A[L']) \to \Fc(\A[L])$ that maps $\Fc(\A[L'])$ identically to $\Fc(\A[L])_{\leq F_0'}$.
 
 We have a commutative diagram
 	\begin{equation} \label{eq:commute-easy1}
		\begin{split}
			\xymatrix{
				\Sc^{F_0'}(\A[F]) 
				\ar[d]_{[G,C]\mapsto[G_{L'},C_{L'}]} \ar[r]^{\iota} & \Sc^{F_0}(\A[F]) \ar[d]^{[G,C]\mapsto[G_{L},C_{L}]}\\
				\Sc(\A[L']) \ar[r]^{j_{F_0'}}& \Sc(\A[L])
			} 
		\end{split}
 	\end{equation}
where $j_{F'_0}$ maps $[G,C] \mapsto [i_{F_0'}(G),i_{F_0'}(C)]$ and the vertical maps are homotopy equivalences (\cite[Proposition 3.3.5]{caldel17}). 

Extending the notation of Definition \ref{df:subcatD} 
we can consider the subdiagram $\mathscr D_{L,F_0'}$
of $\mathscr D$ induced on the subcategory $\Fc(\A^L)$ of $\Fc(\A)$ by the subposets $\mathscr D_{L,F_0'}(F):= \Sc^{F_0'}(\A[F])$. The nerve $\gre{\mathscr D_{L,F_0'}}$
is a subcomplex of both $\Sc_{L',F_0'} = \gre{\GC \mathscr D_{L',F_0'}}$ and $\Sc_{L,F_0} = \gre{\GC \mathscr D_{L,F_0}}$.

The homotopy equivalences referred to in Remark  ~\ref{htprod}  can be naturally extended to $\gre{\mathscr D_{L,F_0'}}$.
Let $\Theta$ be the functor that realizes these homotopy equivalences.
The functor $\Theta$ fits into the commutative diagram \eqref{eq:commute-easy}. The 
left side
commutes because it consists of inclusions. The top-right square commutes trivially from the Definition and the bottom-right square commutes because of Diagram \eqref{eq:commute-easy1}. 
	\begin{equation} \label{eq:commute-easy}
		\begin{split}
			\xymatrix{&& 
				\ar[dll]^{\varphi_{L'}}
				\gre{\GC \mathscr D_{L',F_0'}}
				\ar[r]^(.37){\Theta}_(.37){homeq}
				&
				\gre{\Sc(\Ar{L'})} \times \gre{\Fc(\A^{L'})} \\
				\Sal(\A) 
				&&
				\gre{\int \mathscr D_{L,F_0'}}\ar[u]^{\iota}\ar[d]^{\iota}
								\ar[r]^(.37){\Theta}_(.37){homeq}
				&
				\ar[u]_{\id \times \iota }
				\gre{\Sc(\Ar{L'})\times \gre{ \Fc(\A^{L})}}
				\ar[d]^{ j_{F_0'}\times \id}\\
				&& 
				\ar[ull]^{\varphi_{L}}
				\gre{\GC \mathscr D_{L,F_0}} 
				\ar[r]^(.37){\Theta}_(.37){homeq}
				&
				\gre{\Sc(\Ar{L})}  \times \gre{\Fc(\A^{L})} \\
			}
		\end{split}
	\end{equation}

The horizontal maps of the Diagram \eqref{eq:commute-easy} define isomorphisms in cohomology, we can consider their inverses.	Passing to cohomology we obtain the following commutative diagram:
		\begin{equation} \label{eq:commutecohom}
		\begin{split}
			\xymatrix{
				& & &  H^*(M(\Ar{L'});\Z) \otimes H^*(L';\Z)    \ar[d]^{ \iota^* \otimes \id}\\
				H^*(\Sal(\A) ;\Z) \ar[rrru]^(.40){\varphi_{L'}^*} \ar[rrrd]_{\varphi_L^*}                  & & & H^*(M(\Ar{L'});\Z) \otimes H^*(L;\Z)  \\
				& & & H^*(M(\Ar{L});\Z) \otimes  H^*(L;\Z) \ar[u]_{ j_{F_0'}^*\otimes \id } \\ 
			}
		\end{split}
	\end{equation}
	Now Diagram \ref{eq:commutecohom2} is obtained by  projecting the first factor of the tensor products for the terms in the right side of diagram \eqref{eq:commutecohom} to cohomological degree $q$
	and by taing the direct sum for all $L' \in (\Cc_{\leq L})_q.$ 
	\begin{equation} \label{eq:commutecohom2}
		\begin{split}
			\xymatrix{
				& & & \underset{L' \in (\Cc_{\leq L})_q}{\bigoplus}
				\!\!\! H^q(M(\Ar{L'});\Z) \otimes H^*(L';\Z) \ar[d]^{\id \otimes \iota^*}\\
				H^*(\Sal(\A) ;\Z) \ar[rrru]^(.40){\underset{L' \in (\Cc_{\leq L})_q}{\bigoplus}
					\!\!\! \varphi_{L'}^*} \ar[rrrd]_{\varphi_L^*}                  
				& & & \left( \underset{L' \in (\Cc_{\leq L})_q}{\bigoplus}
				\!\!\! H^q(M(\Ar{L'});\Z)  \right) \otimes H^*(L;\Z)  \\
				& & & H^q(M(\Ar{L});\Z) \otimes H^*(L;\Z) \ar[u]_{\left( \underset{L' \in (\Cc_{\leq L})_q}{\bigoplus}
					\!\!\! j_{F_0'}^* \right) \otimes \id}. \\ 
			}\end{split}
		\end{equation}

		From the Combinatorial Brieskorn Lemma \cite[Proposition 3.3.3]{caldel17} the left factor $\underset{L' \in (\Cc_{\leq L})_q}{\bigoplus}
		\!\!\! j_{F_0'}^* $ of the bottom right map of Diagram \eqref{eq:commutecohom2}
		is an isomorphism. 
		Hence we can invert the bottom right arrow and, composing with the upper right map $\id \otimes \iota^*$ we get the natural vertical map.
		Thus, the commutativity of the diagram of the statement of the Theorem follows.
\end{proof}

\bibliography{erratumbib}{}
\bibliographystyle{alpha}

\end{document}